\documentclass[a4paper,twopage,reqno,12pt]{amsart}
\usepackage[top=30mm,right=30mm,bottom=30mm,left=30mm]{geometry}

\usepackage{graphicx}
\usepackage{amsmath}
\usepackage{amssymb}
\usepackage{amsfonts}
\usepackage{amsthm}
\usepackage{amstext}
\usepackage{amsbsy}
\usepackage{amsopn}
\usepackage{amscd}
\usepackage{enumerate}
\usepackage{url}
\usepackage{color}
\usepackage[colorlinks]{hyperref}
\usepackage[hyperpageref]{backref}
\usepackage{multirow}
\usepackage{rotating}
\usepackage{longtable}
\usepackage{float}
\usepackage{lscape}

\newtheorem{theorem}{Theorem}[section]
\newtheorem{corollary}[theorem]{Corollary}
\newtheorem{lemma}[theorem]{Lemma}
\newtheorem{proposition}[theorem]{Proposition}

\newtheorem{hypothesis}[theorem]{Hypothesis}
\theoremstyle{definition}

\newtheorem{example}[theorem]{Example}

\newtheorem{problem}[theorem]{Problem}

\numberwithin{equation}{section}

\newcommand{\SL}{\mathrm{SL}}

\newcommand{\PSL}{\mathrm{PSL}}

\newcommand{\PSp}{\mathrm{PSp}}
\newcommand{\PGL}{\mathrm{PGL}}
\newcommand{\PGaL}{\mathrm{P\Gamma L}}

\newcommand{\C}{\mathrm{C}}
\newcommand{\D}{\mathrm{D}}

\newcommand{\Q}{\mathrm{Q}}

\newcommand{\A}{\mathrm{A}}
\renewcommand{\S}{\mathrm{S}}
\newcommand{\HA}{\rm{HA}}
\newcommand{\HS}{\rm{HS}}
\newcommand{\HC}{\rm{HC}}
\newcommand{\AS}{\rm{AS}}
\newcommand{\TW}{\rm{TW}}
\newcommand{\SD}{\rm{SD}}
\renewcommand{\CD}{\rm{CD}}
\newcommand{\PA}{\rm{PA}}
\newcommand{\Sym}{\mathrm{Sym}}

\newcommand{\Aut}{\mathbf{Aut}}

\newcommand{\Nbf}{\mathbf{N}}

\newcommand{\fix}{\mathbf{fix}}
\newcommand{\Cbf}{\mathbf{C}}
\newcommand{\Obf}{\mathbf{O}}
\newcommand{\Soc}{\mathbf{Soc}}

\newcommand{\rbf}{\mathbf{r}}
\newcommand{\sbf}{\mathbf{s}}

\newcommand{\M}{\mathrm{M}}

\newcommand{\Dmc}{\mathcal{D}}
\newcommand{\Bmc}{\mathcal{B}}

\newcommand{\Emc}{\mathcal{E}}
\newcommand{\Pmc}{\mathcal{P}}
\newcommand{\Gmc}{\mathcal{G}}

\renewcommand{\leq}{\leqslant}
\renewcommand{\geq}{\geqslant}
\renewcommand{\emptyset}{\varnothing}

\renewcommand{\mod}[1]{\ (\mathrm{mod}{\ #1})}
\newcommand{\imod}[1]{\allowbreak\mkern4mu({\operator@font mod}\,\,#1)}

\begin{document}
 \title{Symmetries of biplanes }

 \author[S.H. Alavi]{Seyed Hassan Alavi}%
 \thanks{Corresponding author: Seyed Hassan Alavi.}
 \address{Seyed Hassan Alavi, Department of Mathematics, Faculty of Science, Bu-Ali Sina University, Hamedan, Iran.
 }%
 \email{alavi.s.hassan@basu.ac.ir and  alavi.s.hassan@gmail.com (G-mail is preferred)}
 \author[A. Daneshkhah]{Ashraf Daneshkhah}%
 \address{Ashraf Daneshkhah, Department of Mathematics, Faculty of Science, Bu-Ali Sina University, Hamedan, Iran.
 }%
  \email{adanesh@basu.ac.ir}
 \author{Cheryl E. Praeger}
 \address{ %
 Cheryl E. Praeger, Centre for the Mathematics of Symmetry and Computation, Department of Mathematics and Statistics, The University of Western Australia, 35 Stirling Highway, Crawley, 6009 W.A.,
 Australia.}
 \email{Cheryl.Praeger@uwa.edu.au}

 \subjclass[]{05B05; 05B25; 20B25}%
 \keywords{Biplane; cartesian decomposition; primitive permutation group}
 \date{\today}%

\begin{abstract}
In this paper, we first study biplanes $\Dmc$ with parameters $(v,k,2)$, where the block size $k\in\{13,16\}$. These are the smallest parameter values for which a classification is not available. We show that if $k=13$, then either  $\Dmc$ is the  Aschbacher biplane or its dual, or  $\Aut(\Dmc)$ is a subgroup of the  cyclic group of order $3$. In the case where $k=16$, we prove that $|\Aut(\Dmc)|$ divides $2^{7}\cdot 3^{2}\cdot 5\cdot 7\cdot 11\cdot 13$. 
We also provide an example of a biplane with parameters $(16,6,2)$ with a flag-transitive and point-primitive subgroup of automorphisms preserving a homogeneous cartesian decomposition. This motivated us to study biplanes with point-primitive automorphism groups preserving a cartesian decomposition. We prove that such an automorphism group is either of affine type (as in the example), or twisted wreath type.
\end{abstract}

\maketitle

\section{Introduction}\label{sec:intro}

A \emph{symmetric design} $\Dmc=(\Pmc,\Bmc)$ with parameters $(v,k,\lambda)$ is an incidence structure consisting of a set $\Pmc$ of $v$ \emph{points} and a set $\Bmc$ of $v$ \emph{blocks}, each of which is a $k$-subset of $\Pmc$, such that every pair of points lies in exactly $\lambda$ blocks. These conditions imply that every point is incident with exactly $k$ blocks, and every pair of blocks intersects in exactly $\lambda$ points. If  $\lambda=1$, then  $\Dmc$ is called a \emph{projective plane}, and if  $\lambda=2$, then  $\Dmc$ is called a  \emph{biplane}.   The aim of the paper is to study the possible symmetries of biplanes.  Whereas there are several infinite families of projective planes, we currently know of only sixteen nontrivial biplanes, and they all have block size $k\leq 13$, as summarised in Table~\ref{tbl:ex}. The column headed ``\# Examples'' contains the number of biplanes up to isomorphism for the given parameters.  In the column headed ``Reference''  we provide some references which give details of the constructions or classification, and in Section~\ref{sec:examples} we discuss briefly some of their properties. In line 7, we list the smallest set of feasible parameters for which a biplane is currently not known to exist.

\begin{table}[!htbp]
    \scriptsize
    \caption{\small Biplanes with small block size}\label{tbl:ex}
\begin{tabular}{lllcll}
    \hline
    Line  & $v$ & $k$ &  \# Examples & Reference  \\
    \hline
    $1$  & $ 7$ & $4$     & $1$ & \cite{a:Kantor-2-trans} \\
    $2$  & $ 11$ & $5$   & $1$ &  \cite{a:Kantor-2-trans}  \\
    $3$  & $ 16$ & $6$   & $3$ & \cite{a:Assmus-79,a:Burau-63,a:Husain} \\
    $4$  & $ 37$ & $9$   & $4$ & \cite{a:Assmus-77} \\
    $5$  & $ 56$ & $11$  & $5$ & \cite{a:Essert-2005,a:Salwach-79,a:Kaski-08} \\
    $6$  & $ 79$ & $13$  & $\geq 2$ & \cite{a:Aschbacher-71,a:Marangunic-92,a:Kaski-08} \\
    $7$  & $ 121$ & $16$ & Unknown &  \\
    \hline
\end{tabular}
\end{table}
Following the exhaustive computer search described in \cite{a:Kaski-08}, we know that the list in Table~\ref{tbl:ex} is complete for block sizes $k\leq 11$. By the Bruck--Ryser--Chowla theorem,  there exists no biplane with $k = 12$. The authors of  \cite{a:Kaski-08} comment that, unless new ideas emerge to make possible an exhaustive search for biplanes with $k=13$, their classification ``is probably not to be seen in the foreseeable future,'' and they note that such searches are more likely to complete if some nontrivial symmetry is assumed.   Part of our aim in this paper is to restrict the symmetries of possible biplanes with block sizes $13$ or $16$. The only known examples are the Aschbacher biplane with $k=13$ constructed in \cite{a:Aschbacher-71} and its dual.  Moreover, it is known from work of Maranguni\'{c} \cite[Theorem 2]{a:Marangunic-92} (building on previous work referenced there) that the full  automorphism group of any other  biplane with parameters $(79, 11,2)$ must be a $3$-group. We restrict this further and also give restrictions for the groups of biplanes with parameters $(121, 16, 2)$.

\begin{theorem}\label{thm:main-121}
Let $\Dmc$ be a biplane with parameters $(v,k,2)$, where $k\in\{13,16\}$. Then
\begin{enumerate}[\rm (a)]
    \item if $k=13$, then either $\Aut(\Dmc)\leq \C_{3}$, or $\Dmc$ is the  Aschbacher biplane or its dual with $|\Aut(\Dmc)|=110$;
    \item if $k=16$, then $|\Aut(\Dmc)|$ divides $2^{7}\cdot 3^{2}\cdot 5\cdot 7\cdot 11\cdot 13$.
\end{enumerate}
\end{theorem}



For our other more general analysis we were motivated by Kantor's  fundamental result \cite[Theorem B]{a:kantor87} which restricted the types of point-primitive automorphism groups of projective planes (and which had such strong influence on future work in that area \cite{a:Camina-2005,a:Gill-2007,a:Gill-2016,a:kantor87}). We hoped that something similar would be available for biplanes.
We observed that some of the known biplanes admit a point-primitive  automorphism group, and  it is possible for such groups to preserve a cartesian decomposition of the point set (as defined in Section~\ref{sec:defn}).
We examined the possible types of primitive permutation groups (as given in Table~\ref{tbl:prim}): for six of the eight types there are some primitive groups of that type which preserve a cartesian decomposition (the exceptions being types $\HS$ and $\SD$).  We showed that groups of only two of these types may act as point-primitive groups on biplanes.

\begin{theorem}\label{thm:main}
Let $\Dmc$ be a biplane admitting a point-primitive automorphism group $G$. If $G$ preserves a cartesian decomposition of the point-set, then $G$ is either of affine or twisted wreath type.
\end{theorem}

Only one of  the known biplanes admits a point-primitive group which preserves a cartesian decomposition. It is one of the biplanes with parameters $(16,6,2)$, and the group is of affine type (see Subsection~\ref{ex:S4wrS2}). Each primitive group of twisted wreath type does indeed preserve a cartesian decomposition, but unfortunately we have not been able to eliminate these groups as automorphism groups of biplanes, even though there are no known examples.

\begin{problem}\label{ques:TW}
Show that no biplane  admits a point-primitive automorphism group of twisted wreath type.
\end{problem}


This paper is organised as follows. In Section \ref{sec:examples}, we provide detailed information about the symmetry properties of the known biplanes. In particular, in Example~\ref{ex:S4wrS2}, we give an example of biplane with parameters $(16,6,2)$ admitting a point-primitive automorphism group preserving a homogeneous cartesian decomposition. In Section \ref{sec:fix} we summarise existing results and develop new ones concerning fixed points of automorphisms of biplanes. These are used for studying automorphisms  of biplanes with parameters $(79,13,2)$ in Section~\ref{sec:79}, and those  with parameters $(121,16,2)$ in Section~\ref{sec:121}, in particular  proving Theorem \ref{thm:main-121}.
We study transitive biplanes with point-primitive automorphism groups in product action in Section~\ref{sec:ProdAction}, and in particular, we show in Theorem \ref{thm:main-pa}  that no transitive biplane has a primitive automorphism group of type \PA \ (as defined in Table~\ref{tbl:prim}). As mentioned above, the example in Subsection~\ref{ex:S4wrS2} motivated us to investigate biplanes with an automorphism group preserving a homogenous cartesian decomposition 
of the point set $\Pmc = \Gamma^d$ with $|\Gamma|\geq 5$. Such groups can be embedded into the primitive wreath product $\Sym(\Gamma)\wr \S_{d}$ in product action \cite[Theorem~5.13]{b:Praeger-18}, and using  Praeger's result \cite{a:prim-incl} on  the overgroups of finite primitive permutation groups, we analyse the possible cases to prove Theorem~\ref{thm:main}.

\begin{table}
    \scriptsize
    \caption{\small Types of primitive groups $G$ on a set $\Omega$. Here $T$ is a nonabelian simple group and $\alpha\in \Omega$}\label{tbl:prim}
\begin{tabular}{llp{9cm}}
    \hline
    Name & O'Nan-Scott type & Description\\
    \hline
    \HA &
    Affine &
    The unique minimal normal subgroup is elementary abelian and acts regularly on $\Omega$\\
    \HS &
    Holomorph simple &
    Two minimal normal subgroups: each nonabelian simple and regular\\
    \HC &
    Holomorph compound  &
    Two minimal normal subgroups: each isomorphic to $T^{k}$, $k\geq 2$ and regular\\
    \AS &
    Almost simple &
    The unique minimal normal subgroup is nonabelian simple\\
    \TW &
    Twisted wreath &
    The unique minimal normal subgroup is regular and isomorphic to $T^{k}$ with $k\geq 6$\\
    \SD &
    Simple diagonal &
    The unique minimal normal subgroup $N$ is isomorphic to $T^{k}$, with $k\geq 2$ and $N_{\alpha}$ is a full diagonal subgroup of $N$. The group $G$ acts primitively on the set of $k$ simple direct factors of $N$\\
    \CD &
    Compound diagonal &
    The unique minimal normal subgroup $N$ is isomorphic to $T^{k}$ and, for some $\ell\mid k$ with $k\geq 2\ell\geq 4$,
    $N_{\alpha}\cong T^{\ell}$ is a direct product of `strips', each a diagonal subgroup of $T^{k/\ell}$. The support sets of the strips
    form a system of imprimitivity for the $G$-action on the set of $k$ simple direct factors of $N$\\
    \PA &
    Product action &
    The unique minimal normal subgroup $N$ is isomorphic to $T^{k}$,with $k\geq 2$ and $N_{\alpha}\cong R^{k}$ for some proper nontrivial subgroup $R$ of $T$. \\
    \hline
\end{tabular}
\end{table}

\section{Definitions and notation}\label{sec:defn}

All groups and incidence structures in this paper are finite. We here adopt the standard notation as in \cite{b:Atlas} for finite groups of Lie type, for example, we use $\SL_{n}(q)$ and $\PSL_{n}(q)$ to denote special linear groups and projective spacial linear groups, respectively. We write $\C_{n}$ for the cyclic group of order $n$. 
Symmetric and alternating groups on $\Omega$ are denoted by $\S_{n}$ and $\A_{n}$, respectively. 

Let $\Omega=\{1,\ldots,n\}$, and let $G$ be a group acting on $\Omega$ with $\alpha\in \Omega$. 
We denote the \emph{$G$-orbit} containing $\alpha$ by $\alpha^{G}$ and denote the \emph{point-stabiliser} of $\alpha$ by $G_{\alpha}$. A group $G$ is \emph{transitive} on $\Omega$ if it has only one orbit, and it acts \emph{semiregularly} on $\Omega$ if $G_{\alpha}=1$, for all $\alpha \in \Omega$.  If $G$ is both transitive and semiregular, then it is called \emph{regular}. 
If the identity element of $G$ is the only element of $G$ which fixes every point, then the action of $G$ on $\Omega$ is called \emph{faithful} and $G$ may be viewed as a subgroup of $\S_n$; in this case, $G$ is called a \emph{permutation group} on $\Omega$.   
A \emph{Frobenius group} is a transitive permutation group which is not regular and the identity element is the only element of $G$ fixing  more than one point. 
If $\Delta$ is a nonempty subset of $\Omega$, we denote by $\fix_{\Delta}(x)$ the set of all points in $\Delta$ fixed by $x$. We also write $G_{\Delta}$ for the  setwise-stabiliser of $\Delta$. The group $G^{\Delta}$ denotes the \emph{induced permutation group} of $G$ on $\Delta$. A transitive group $G$ on $\Omega$ is said to be \emph{imprimitive} if there is a nontrivial $G$-invariant partition of $\Omega$, and otherwise is said to be \emph{primitive}. Table~\ref{tbl:prim} describes the various  types of primitive permutation groups occurring in the O'Nan-Scott Theorem \cite{a:LPS-onan-88} giving also the notation from \cite{a:prim-incl}.

Let $\Dmc=(\Pmc,\Bmc)$ be a symmetric design  with parameters $(v,k,\lambda)$.  For $\alpha\in\Pmc$ and $B\in\Bmc$, we denote by $[\alpha]$ the set of blocks containing $\alpha$, and  by $[B]$ the  set of points incident with $B$.  A \emph{flag} of $\Dmc$ is an incident point-block pair.
From the definition of $\Dmc$ in the introduction, it is clear that $(\Bmc,\Pmc)$ is also a symmetric design  with parameters $(v,k,\lambda)$
(where we identify each $\alpha \in\Pmc$ with the subset $[\alpha]$ of $\Bmc$); the design $(\Bmc,\Pmc)$ is called the \emph{dual design} of $\Dmc$. Recall that a \emph{biplane} is a symmetric design with $\lambda=2$.
According to our definition (where blocks are defined as $k$-subsets of points) there are no biplanes with $k=2$ since each pair of distinct points would be equal to the unique block containing the two points. Also,  if $k=3$, then $v=4$ and $\Dmc$ is the complete design with block-set the set of all $3$-subsets of points. Therefore, we will assume that $k\geq 4$ and in this case $v\geq 7$ (which follows from Lemma~\ref{lem:basic}).

An automorphism of $\Dmc=(\Pmc,\Bmc)$ is a permutation of the point-set $\Pmc$  which leaves invariant the block-set $\Bmc$.
 A subgroup $G$ of the full automorphism group acts naturally on the sets of points, blocks and flags, and the notion of transitivity and primitivity of these actions is defined using the same terminology as for permutation groups.  We remark here that the conditions of  point-transitivity and block-transitivity are equivalent for symmetric designs \cite{a:Block-67} and in particular for biplanes, and therefore, in this paper, we refer to a design with these properties simply as a \emph{transitive} design. For an automorphism $x$ of $\Dmc$, we denote the sets of fixed points and fixed blocks of $x$ by $\fix_{\Pmc}(x)$ and $\fix_{\Bmc}(x)$, respectively. Also for a subset $X\subset\Pmc$, we often write $\fix_{X}(x)=\fix_{\Pmc}(x)\cap X$ and similarly for subsets of $\Bmc$. Note that a symmetric design and its dual design have isomorphic automorphism groups.

A \emph{cartesian decomposition}  of the set $\Pmc$ is a finite set $\Emc=\Gamma_{1}$, \ldots, $\Gamma_{d}$ of partitions of $\Pmc$ such that $|\Gamma_{i}|\geq 2$ for each $i$ and
$|\gamma_{1}\cap \cdots \cap \gamma_{d}|=1$ for  each choice of  $\gamma_{1}\in \Gamma_{1},\dots,\gamma_{d}\in \Gamma_{d}$.
A cartesian decomposition is said to be \emph{trivial} if it contains only one partition, (in this case it must be the partition into singletons), and it is called \emph{homogeneous} if all the $\Gamma_{i}$ have the same cardinality.

We use the computer software system \textsf{GAP} \cite{GAP4} for computational arguments. For further notation and definitions on designs and permutation groups see \cite{b:Atlas,b:Dixon,b:Hugh-design,b:lander,b:Praeger-18}.

\section{Symmetries of the known biplanes}\label{sec:examples}

Here we give a brief summary of the known examples of biplanes. These all have  block size $k\leq 13$.  The parameters for these biplanes are given in Table \ref{tbl:ex} along with references where more details can be found about them.. \smallskip

\noindent \textbf{Line 1:} The unique biplane with parameters $(7,4,2)$  is the complement of the Fano plane (the unique projective plane  with parameters $(7,3,1)$). Its automorphism group is $\PGL_3(2)\cong \PSL_{2}(7)$ and is flag-transitive and point-primitive. \smallskip

\noindent \textbf{Line 2:} The unique biplane with parameters $(11,5,2)$ is the Hadamard design,  see \cite{a:Kantor-2-trans}. Its automorphism group is $\PSL_{2}(11)$, and it is $2$-transitive and hence primitive on points, and also flag-transitive.\smallskip

\noindent \textbf{Line 3:} In the case of  biplanes with parameters $(16,6,2)$, Hussain \cite{a:Husain} proved that there are exactly three examples. This classification was obtained independently by Burau \cite{a:Burau-63} who determined all three automorphism groups (see \cite[p. 263]{a:Assmus-79}). O'Reilly Regueiro \cite[p. 139]{a:reg-reduction} described these three biplanes and their symmetry properties.  All three are transitive.  One of them has a point-primitive and flag-transitive full automorphism group $2^{4}{:}\S_{6}$, and we have constructed a point-primitive and flag-transitive subgroup which  preserves a homogeneous cartesian decomposition of the  points set, see Subsection~\ref{ex:S4wrS2}.
The other two of these biplanes have point-imprimitive (full) automorphism groups. One of them arises from a difference set in $\C_{2} \times \C_{8}$. Its  full automorphism group is flag-transitive with a point-stabiliser of order $48$ of the form $2{\cdot}\S_4$ (the full group of symmetries of the cube).  The third biplane arises from a difference set in $\Q_{8} \times \C_{2}$. Its full automorphism group has order $16\cdot  24$, and is not flag-transitive. \smallskip

\noindent \textbf{Line 4:}
We know from \cite{a:Salwach-78} that there are exactly  four biplanes with parameters $(37,9,2)$, see also \cite{a:Assmus-77}.
Only one of these biplanes is transitive. It is  constructed from the difference set of nine quadratic residues modulo $37$, and is flag-transitive.
Two of the remaining three biplanes are duals of each other and were found by Hussain \cite{a:Hussain-45}. Their  full automorphism group is $\PGaL_{2}(8)$ and is intransitive. The last of these biplanes has an automorphism group of order $54$, which fixes a unique point. \smallskip

\noindent \textbf{Line 5:}
We know from \cite{a:Kaski-08} that there are exactly  five biplanes with parameters $(56,11,2)$. Discussions of them may be found in   \cite{a:Assmus-79,a:Denniston-80,a:Essert-2005,a:Hall-70,a:Key-97,a:Salwach-79}. Only one of these biplanes  is transitive. It was
found by Hall et al. \cite{a:Hall-70} in the study of rank 3 permutation groups, and has  full automorphism group
$\PSL_{3}(4){:}\C_{2}^{2}$ (a subgroup of index $3$ of $\Aut(\PSL_{3}(4))$).  This group is  point-primitive but not flag-transitive with a
point-stabiliser $\A_{6}{\cdot} \C_{2}^2$.  
\smallskip

\noindent \textbf{Line 6:}
In the case of  biplanes with parameters $(79,13,2)$, there are two known examples, namely the  Aschbacher biplane and its dual \cite{a:Aschbacher-71}.  The full automorphism group of the Aschbacher biplane and its dual is
\[
\langle x,y,z \mid x^{11}=y^5=z^2=1, x^y=x^4, x^z=x^{-1}, yz=zy\rangle
\]
of order $110$ and is intransitive, see \cite[Theorem 6.9]{a:Aschbacher-71}. Further, it is known from work of Maranguni\'{c} \cite[Theorem 2]{a:Marangunic-92} (building on previous work referenced there) that the full  automorphism group of any other  biplane with these parameters must be a $3$-group. In fact the group must have order $1$ or $3$ by Theorem~\ref{thm:main-121}.
 \smallskip

\noindent \textbf{Line 7:}
There are no known biplanes with parameters $(121,16,2)$.  However, if such a biplane exists, then  its full automorphism group has order dividing  $2^{7}\cdot 3^{2}\cdot 5\cdot 7\cdot 11\cdot 13$, by  Theorem~\ref{thm:main-121}, and is not transitive by Corollary~\ref{cor:p2}.

\section{Fixed points of biplane automorphisms}\label{sec:fix}

In this section, we collect information about the fixed points of automorphisms of biplanes. First we summarise some results about their parameters.

\begin{lemma}\label{lem:basic}
Let $\Dmc = (\Pmc,\Bmc)$ be a biplane with parameters $(v,k,2)$ and $k\geq4$. 
Then
\begin{align}\label{eq:basic}
    k(k-1)=2(v-1).
\end{align}
If $v=c^d$, for  positive integers $c,d$ with $d\geq2$, then
\begin{align}\label{eq:kbound}
    \sqrt{2}\,c^\frac{d}{2}<k<\sqrt{2}\,c^\frac{d}{2}+1,
\end{align}
and so $k=\lceil \sqrt{2}c^\frac{d}{2} \rceil$ and if $d$ is odd, then $c\neq 2^a$ with $a$ odd. In particular, if $d=2$, then $k=\lceil \sqrt{2}c \rceil$ with $c\geq 4$.
\end{lemma}

\begin{proof}
The equation \eqref{eq:basic} can be found in \cite[Proposition 1.1]{b:lander}.
Suppose now that $v=c^d$ with $d\geq2$. Then  \eqref{eq:basic} implies that $(k-1)^{2}<k(k-1)=2c^{d}-2$, and so $k<\sqrt{2c^{d}-2}+1<\sqrt{2}c^\frac{d}{2}+1$. Also, by \eqref{eq:basic}, $2v=k^2-k+2 < k^2$ (since $k\geq4$), so $k>\sqrt{2} c$, and \eqref{eq:kbound} is proved.
If $d$ is odd and $c=2^a$ with $a$ odd, then  $s:=2^{a(d+1)/2}$ is a positive integer and the inequality \eqref{eq:kbound} is $s< k <s+1$, which has no integer solution $k$. Thus these conditions do not hold, and so  $\sqrt{2}c^\frac{d}{2}$ is an irrational number, and   $k=\lceil \sqrt{2}c^\frac{d}{2} \rceil$. In particular, if $d=2$, then $k=\lceil \sqrt{2}c \rceil$; in this case, if $c=2$ or $3$, then $2(v-1)=2(c^d-1)=6$ or $16$, and $k=3$ or $5$, respectively, and this contradicts either the assumption $k\geq4$ or the condition  \eqref{eq:basic}. Thus $c\geq 4$.
\end{proof}


\begin{lemma}\cite[Theorem 3.1 and Corollary 3.2]{b:lander}\label{lem:fix}
An automorphism $x$ of a symmetric
design fixes an equal number of points and blocks, and moreover, $x$ has the same cycle structure in its actions on points and on blocks.
\end{lemma}

In the light of Lemma~\ref{lem:fix}, for  an automorphism $x$ of a symmetric design $\Dmc=(\Pmc,\Bmc)$, we refer to the
cycle structure of the $x$-action on $\Pmc$, or on $\Bmc$, as the \emph{cycle structure} of $x$.

A pair $\Dmc'=(\Pmc',\Bmc')$ is called a \emph{subdesign} of a $2$-design $\Dmc=(\Pmc,\Bmc)$ if $\Pmc' \subseteq \Pmc$, each  $B'\in\Bmc'$ is of the form $B'=B\cap\Pmc'$ for some $B\in \Bmc$, and $\Dmc'$ is a $2$-design.  Note that for $\Dmc'$ to be a subdesign, there must be a constant $k'\geq2$ such that, for each $B\cap\Pmc'\in\Bmc'$, the size $|B\cap \Pmc'|=k'$. The subdesign $\Dmc'$ is called \emph{proper} if $\Pmc'\ne\Pmc$.


\begin{lemma}\cite[Lemma 9.5]{a:Kantor-2-trans}\label{lem:subdesign}
Let $\Dmc$ be a symmetric design with parameters $(v,k,\lambda)$. Suppose that $\Dmc'$ is a proper subdesign of $\Dmc$, and that $\Dmc'$ is a symmetric design with parameters $(v',k',\lambda)$. Then either $(k'-1)^{2}=k-\lambda$, or $k'(k'-1)\leq k-\lambda$. (The possibility that $k'=\lambda+1$ is not excluded.)
\end{lemma}

\begin{proof}
It follows from \cite[Lemma 9.5]{a:Kantor-2-trans} that either $(k'-1)^{2}=k-\lambda$, or $\lambda v'\leq k$, so assume that the latter holds.   Then $\lambda(v'-1)\leq k-\lambda$, and since $\Dmc'$ is a symmetric design we have
$k'(k'-1)=\lambda(v'-1)$, see \cite[Proposition 1.1]{b:lander}.
\end{proof}

\begin{lemma}\cite[Lemma 2.7]{a:Aschbacher-71}\label{lem:Autprims}
Let $\Dmc$ be a symmetric design with parameters $(v,k,\lambda)$, and let $p$ be a prime dividing $|\Aut(\Dmc)|$.
Then either $p$ divides $v$, or $p\leq k$.
\end{lemma}

Recall that, for a design $\Dmc=(\Pmc,\Bmc)$ and $\alpha\in\Pmc$, $B\in\Bmc$, we denote by $[\alpha], [B]$  the set of blocks containing $\alpha$, or the set of points contained in $B$, respectively.

\begin{lemma}\label{lem:fix-bi}
Let $\Dmc=(\Pmc,\Bmc)$ be a  biplane with parameters $(v,k,2)$ and $k\geq4$, and let $x\in \Aut(\Dmc)$ such that $\alpha \in \fix_\Pmc(x)$, so  there is a block $B \in \fix_\Bmc(x)$ by {\rm Lemma \ref{lem:fix}}. Then, setting $\sbf_{\alpha} = |\fix_{[\alpha]}(x)|$
and $\sbf_{B}=|\fix_{[B]}(x)|$, the following hold:
\begin{enumerate}[\rm (a)]
    \item if $\alpha\in B$, then $\rbf_{\alpha}=\rbf_{B}$, where $\rbf_{\alpha}$ is the number of
    $\langle x\rangle$-orbits of length two in $[\alpha]$, and  $\rbf_{B}$ is the number of $\langle x\rangle$-orbits of length two in $[B]$;
    \item $|\fix_\Bmc(x)|=\sbf_{B}(\sbf_{B}-1)/2+\rbf_{B}+1$;
    \item if $\alpha\in B$, then   $|\fix_\Pmc(x)|=\sbf_{\alpha}(\sbf_{\alpha}-1)/2+\rbf_{\alpha}+1$;
    \item if $\sbf_{B}>0$ and $\rbf_{B}=0$, then $\Dmc'=(\fix_\Pmc(x), \fix_\Bmc(x))$ is a subdesign  of $\Dmc$, and $\sbf_B$ is independent of the choice of $B\in\fix_\Bmc(x)$; moreover, either $(\sbf_{B}-1)^{2}=k-2$, or $\sbf_{B}(\sbf_{B}-1)\leq k-2$.
    \end{enumerate}
\end{lemma}

\begin{proof}
Parts (a)-(b) are proved in \cite[Lemma 4.1]{a:Aschbacher-71}. For part (c), assume that $\alpha\in B$, and consider the following map $f$ from points of $\fix_\Pmc(x)\setminus\{\alpha\}$ to unordered pairs of blocks from $[\alpha]$:
\[
f:\ \beta\quad \to\quad  \{B'\in \Bmc \mid \{\alpha,\beta\}\subseteq B'\}.
\]
The map $f$ is injective since $\Dmc$ is a biplane. Also each $f(\beta)$ is either an unordered pair of blocks from $\fix_{[\alpha]}(x)$ (and each unordered pair of such blocks intersects in $\{\alpha,\beta\}$ with $\beta\in \fix_\Pmc(x)$), or a pair of blocks from $[\alpha]$ which forms an  $\langle x\rangle$-orbit of length two (and each such orbit also   intersects in $\{\alpha,\beta\}$ with $\beta\in \fix_\Pmc(x)$). 
Thus we have $|\fix_\Pmc(x)|-1 = \sbf_{\alpha}(\sbf_{\alpha}-1)/2+\rbf_{\alpha}$, proving part (c).

Finally, for part (d), assume that $\sbf_B>0$ and $\rbf_B=0$. Then by \cite[Lemma 4.1(iv)]{a:Aschbacher-71}, the incidence structure
$\Dmc'$ is a subdesign  of $\Dmc$. This means first that  $\sbf_B$ is independent of the choice of $B\in\fix_\Bmc(x)$, and secondly that distinct blocks  from $\fix_\Bmc(x)$ intersect in $\lambda=2$ points of $\fix_\Pmc(x)$. The remaining assertions follow from Lemma~\ref{lem:subdesign}.
\end{proof}

\begin{lemma}\label{lem:fix-v}
Let $\Dmc=(\Pmc,\Bmc)$ be a  biplane with parameters $(v,k,2)$ and $k\geq4$, and let $x\in \Aut(\Dmc)$ such that $|\fix_\Bmc(x)|>0$, and $x$ has no $2$-cycles in its cycle structure.  Then  the sets $[B]\setminus \fix_{[B]}(x)$ are pairwise disjoint, for $B\in \fix_\Bmc(x)$,
$\sbf=|\fix_{[B]}(x)|$ is independent of $B\in \fix_\Bmc(x)$, and  $v\geq |\fix_{\Bmc}(x)|\cdot (k-\sbf+1)$.
\end{lemma}

\begin{proof}
If $|\fix_\Bmc(x)|=1$, then  there is nothing to prove. So assume that $|\fix_\Bmc(x)|\geq2$.
Let $A, B\in \fix_\Bmc(x)$ with $A\ne B$, so $A\cap B$ is an $x$-invariant subset of size $2$. Since $x$ has no $2$-cycles in its cycle structure, it follows that $A\cap B\subseteq \fix_\Pmc(x)$. Thus $[B]\setminus \fix_{[B]}(x)$ is disjoint from $[A]\setminus \fix_{[A]}(x)$, and each of these subsets is contained in $\Pmc\setminus\fix_\Pmc(x)$. Further, since $A\cap B\subseteq \fix_\Pmc(x)$, it follows that
$\sbf_B=|\fix_{[B]}(x)|\geq 2$, and by our assumption on $x$,   the number $\rbf_{B}$ of $\langle x\rangle$-orbits of length two in $[B]$
is zero. Thus, by Lemma~\ref{lem:fix-bi}(d),   $\sbf_B$ is independent of  $B\in\fix_\Bmc(x)$, say  $\sbf=\sbf_B$. Then
\[
v\geq |\fix_{\Pmc}(x)|+\sum_{B\in \fix_{\Bmc}(x)} |[B]\setminus \fix_{[B]}(x)| =   |\fix_{\Pmc}(x)| +  |\fix_{\Bmc}(x)| (k-\sbf),
\]
and since  $|\fix_{\Pmc}(x)|=|\fix_{\Bmc}(x)|$ by Lemma~\ref{lem:fix}, the required inequality follows.
\end{proof}

We also consider involutory automorphisms, that is, those of order $2$.

\begin{lemma}\label{lem:bi-2}
Let $\Dmc=(\Pmc,\Bmc)$ be a  biplane with parameters $(v,k,2)$ and $k\geq4$, and let $x\in \Aut(\Dmc)$ of order $2$ such that
 $|\fix_\Pmc(x)|>0$. Let  $\alpha\in \fix_\Pmc(x)$,  $B\in \fix_\Bmc(x)$, and let $\sbf_{\alpha}=|\fix_{[\alpha]}(x)|$ and $\sbf_B=|\fix_{[B]}(x)|$. Then
\begin{enumerate}[\rm (a)]
    \item either  $\sbf_{\alpha}$  is independent of  $\alpha\in \fix_\Pmc(x)$, or $\sbf_{\alpha}\in\{0,2\}$ for all  $\alpha\in \fix_\Pmc(x)$; in both cases $|\fix_\Pmc(x)|=(k+1+(\sbf_{\alpha}-1)^{2})/2$;
    \item either $\sbf_{B}$  is independent of $B\in \fix_\Bmc(x)$, or $\sbf_{B}\in\{0,2\}$ for all  $B\in \fix_\Bmc(x)$;
    \item either $(\sbf_{\alpha}-2)^{2}=k-2$, or $(\sbf_{\alpha}-1)^{2}\leq k-5$. 
\end{enumerate}
\end{lemma}
\begin{proof}
Part (b) and the last assertion of part (a) are proved in \cite[Lemma 4.2]{a:Aschbacher-71}; the rest of part (a) follows on equating the expressions for $|\fix_\Pmc(x)|$ in part (a) for different points in $\fix_\Pmc(x)$. Finally part (c) is  \cite[Theorem~3.3]{a:Arasu-93}.
\end{proof}

We put together information from the last few lemmas in the following corollary, and use the notation $\sbf_\alpha, \sbf_B$ from there.

\begin{corollary}\label{cor:fix-bi-bound}
Let $\Dmc=(\Pmc,\Bmc)$ be a  biplane with parameters $(v,k,2)$ and $k\geq4$, let $x\in \Aut(\Dmc)$, and let $f:=|\fix_{\Pmc}(x)|=|\fix_{\Bmc}(x)|$. Then
\begin{enumerate}[\rm (a)]
    \item if $x$ is an involution and $\alpha\in \fix_\Pmc(x)$, then either $f\leq k-2$, or $k-2$ is a square and $f=k+\sqrt{k-2}=k+\sbf_{\alpha}-2$;
    \item if $o(x)$ is odd and $B\in \fix_\Bmc(x)$, then $f=\sbf_{B}(\sbf_{B}-1)/2+1$; further, either $f\leq k/2$, or  $k-2$ is a square and $f=(k+\sqrt{k-2})/2=(k+\sbf_{B}-1)/2$.
\end{enumerate}
Moreover, if $1\ne G\leq\Aut(\Dmc)$, then  $|\fix_{\Pmc}(G)|\leq k+\sqrt{k-2}\leq (\sqrt{2}+4)k/4$, and if $k\geq 6$, then $|\fix_{\Pmc}(G)|\leq 4k/3$.
\end{corollary}

\begin{proof}
(a) Suppose first that $o(x)=2$ and apply Lemma~\ref{lem:bi-2} parts (a) and (c). Suppose first that 
$\sbf_\alpha=1$, for some $\alpha\in\fix_\Pmc(x)$. Then by Lemma~\ref{lem:bi-2}(c), $k\geq 5+(\sbf_\alpha-1)^2=5$, (since $k-2\geq 2 >(\sbf_\alpha-2)^2$); and then Lemma~\ref{lem:bi-2}(a) gives $f=(k+1)/2\leq k-2$. Suppose next that $\sbf_\alpha=0$, for some $\alpha\in\fix_\Pmc(x)$. Then by Lemma~\ref{lem:bi-2}(c), $k\geq 6$ and by Lemma~\ref{lem:bi-2}(a), $f=(k+2)/2\leq k-2$. Thus, we may assume that $\sbf_\alpha\geq2$, for each  $\alpha\in\fix_\Pmc(x)$. 
Suppose next that 
 $k-2=(\sbf_{\alpha}-2)^{2}$ (a square). Then $\sbf_\alpha\geq 4$ since $k\geq4$, and
\[
(\sbf_{\alpha}-1)^{2}= (\sbf_{\alpha}-2)^{2} +2(\sbf_{\alpha}-2)+1 = k-1 +2 \sqrt{k-2}.
\]
Hence,  by Lemma~\ref{lem:bi-2}(a), $f=(k+1+(\sbf_{\alpha}-1)^{2})/2=k+\sqrt{k-2}=k+\sbf_{\alpha}-2$. If this is not the case, then by Lemma~\ref{lem:bi-2}(c), we have   $(\sbf_{\alpha}-1)^{2}\leq k-5$, and hence  $f=(k+1+(\sbf_{\alpha}-1)^{2})/2\leq k-2$. 

(b)  Now assume that $o(x)$ is odd and $B\in \fix_\Bmc(x)$, and note that Lemma~\ref{lem:fix-bi}(b) applies with $\rbf_B=0$, yielding  $f=\sbf_{B}(\sbf_{B}-1)/2+1$. If $\sbf_B=0$, then $f<k/2$, so assume that $\sbf_B>0$. Then by Lemma~\ref{lem:fix-bi}(d),  either
$\sbf_{B}(\sbf_{B}-1)\leq k-2$ which implies that $f\leq k/2$, or $k-2=(\sbf_{B}-1)^{2}$ (a square) and
\[
f=\sbf_{B}(\sbf_{B}-1)/2+1=[\sqrt{k-2}\cdot (\sqrt{k-2}+1)+2]/2 =(k+\sqrt{k-2})/2.
\]
Thus part (b) is proved.

Note that, by parts (a) and (b), we always have $f\leq k+\sqrt{k-2}$, and since $k-2\leq k^2/8$, for all $k\geq4$, this gives
$f\leq k+\sqrt{k-2}<k+(k/\sqrt{8})=(\sqrt{2}+4)k/4$, while if $k\geq 6$, then $k-2\leq k^2/9$, so $f\leq k+\sqrt{k-2}\leq 4k/3$.
The final assertions follow since  $|\fix_{\Pmc}(G)|\leq  |\fix_{\Pmc}(x)|$, for all non-identity $x\in G$.
\end{proof}


\begin{lemma}\label{lem:conj}
\cite[Lemma 2.3]{a:CNP-alt-linear}
Let $G$ be a group acting transitively on a set $\Omega$, let $\alpha\in \Omega$, and let $x\in G$ such that $f:=|\fix_{\Omega}(x)|>0$. Then $|\Omega|/f=u/u_{1}$, where $u$ is the number of $G$-conjugates of $x$, and $u_{1}$ is the number of $G$-conjugates of $x$ lying in the stabiliser $G_{\alpha}$. In particular, $|\Omega|/f\leq |G:\Cbf_{G}(x)|$, where $\Cbf_{G}(x)$ is the centraliser of $x$ in $G$.
\end{lemma}


\begin{corollary}\label{cor:k-upper}
Let $\Dmc=(\Pmc,\Bmc)$ be a  biplane with parameters $(v,k,2)$ and $k\geq4$, and suppose that $G\leq \Aut(\Dmc)$ is transitive on $\Pmc$.  If some non-identity element $x\in G$ fixes a point, then $k< (\sqrt{2}+4)u/2+1<3u+1$, where $u$ is the number of conjugates of $x$ in $G$.
\end{corollary}
\begin{proof}
By assumption, $f=|\fix_\Pmc(x)|>0$, so by Corollary~\ref{cor:fix-bi-bound} and \eqref{eq:basic},  $f\leq  (\sqrt{2}+4)k/4=(\sqrt{2}+4)(v-1)/(2k-2)$, and hence  $k-1\leq (\sqrt{2}+4)(v-1)/(2f)<(\sqrt{2}+4)v/(2f)$. Since $G$ is point-transitive and $f>0$,  the number $u_{1}$ of $G$-conjugates of $x$ in $G_{\alpha}$ is at least $1$. Hence by  Lemma~\ref{lem:conj}, $k-1 < (\sqrt{2}+4)v/(2f)\leq (\sqrt{2}+4)u/2 < 3u$.
\end{proof}

\subsection{Biplanes and difference sets}


Let $v, k, \lambda$ be integers such that $0<\lambda < k < v$. A \emph{$(v,k,\lambda)$-difference set} in a group $X$ of order $v$ is a subset $D\subset X$ of size $k$ such that the list of `differences' $x^{-1}y$, for distinct $x, y\in D$, contains each non-identity element of $G$ exactly $\lambda$ times. The \emph{development of $D$} is the pair $(X, \Bmc)$ where $\Bmc=\{ Dx\mid x\in X\}$. By
\cite[Theorem 4.1]{b:lander}, the development of $D$ is a symmetric design with parameters  $(v,k,\lambda)$ and the group $X$ acts by right multiplication as a point-regular subgroup of automorphisms. Thus a $(v,k,2)$-difference set gives rise, via its development, to a transitive biplane with parameters $(v,k,\lambda)$.  

\begin{proposition}\label{cor:p2}
Let $\Dmc=(\Pmc,\Bmc)$ be a  biplane with parameters $(p^2,k,2)$ where $k\geq4$ and $p$ is a prime. Then $p\geq 11$ and if $x\in \Aut(\Dmc)$ has order $p$, then $\fix_{\Pmc}(x)=\emptyset$. Moreover, if $G\leq \Aut(\Dmc)$ is transitive on $\Pmc$, then the Sylow $p$-subgroups of $G$ are regular and abelian of order $p^{2}$, and $\Dmc$ is the development of a difference set with parameters $(p^2,k,2)$. 
\end{proposition}

\begin{proof}
 Note that $p\geq 11$ by Table~\ref{tbl:ex}. By  \eqref{eq:basic}, $k$ divides $2(v-1)=2(p^{2}-1)$, and hence $k$ is coprime to $p$.
Suppose that $x\in \Aut(\Dmc)$ has order $p$ and $\alpha\in\fix_{\Pmc}(x)$. Then by Lemma~\ref{lem:fix}, $x$ fixes a block, say $B$.
Let $\sbf_{B}=|\fix_{[B]}(x)|$ and $f=|\fix_{\Bmc}(x)|$. Then $p$ divides $k-\sbf_{B}$ and $p$ divides $f$, and since  $k$ is coprime to $p$ it follows that $\sbf_{B}>0$. Since $p$ is odd, it follows from Lemma~\ref{lem:fix-bi}(b) that
\begin{align}\label{eq:p2}
    2f=\sbf_{B}(\sbf_{B}-1)+2.
\end{align}
If  $\sbf_{B}=k$, then \eqref{eq:p2} implies that $f=(k^2-k+2)/2=v$, which is impossible. Hence $\sbf_B<k$. Thus $k-\sbf_{B}\geq p$ and $f\geq p$ (since $p$ divides both of these quantities), and hence by Lemma~\ref{lem:fix-v},  $p^{2}=v\geq f(k-\sbf_{B}+1)\geq p(p+1)$, which is a contradiction. Hence $\fix_{\Pmc}(x)=\emptyset$.

Now suppose that  $G\leq \Aut(\Dmc)$ is transitive on $\Pmc$, and let $P$ be a Sylow $p$-subgroup of $G$, so $|P|=p^a\geq p^2$. It follows from the previous paragraph that $|P|=p^2$ and that $P$ is regular of $\Pmc$. Thus $P$ is abelian of order $p^2$, and as discussed above,
by \cite[Theorem~4.2]{b:lander}, $\Dmc$ is the development of a difference set in $P$ with parameters $(p^2,k,2)$.
\end{proof}

By using the following result (Lemma~\ref{lem:mann}) of Lander,  we apply Proposition~\ref{cor:p2} to show that no transitive biplanes with parameters $(121,16,2)$ exist. 

\begin{lemma}\label{lem:mann} \cite[Theorem~4.4]{b:lander}
Suppose that there exists a $(v,k,\lambda)$-difference set $\Dmc$ in a group $G$. If for some divisor $p>1$ of $v$ and some prime $q$, there exists an integer $j$
such that $q^{j} \equiv -1 \mod p$, then $q$ does not divide the square-free part of $k-\lambda$.

\end{lemma}

\begin{corollary}\label{cor:121-t}
There exist no transitive biplanes with parameters   $(121, 16, 2)$.
\end{corollary}
\begin{proof}
If $\Dmc$ were such a biplane, then by Proposition \ref{cor:p2}, $\Dmc$ would be 
the development of a difference set with parameters $(121, 16, 2)$. Since $2^{5}\equiv -1 \mod{11}$, we may apply Lemma \ref{lem:mann} with $(p, q, j, k-\lambda) = (11, 2, 5, 14)$ and  conclude that $2$ does not divide the square-free part of $14=2\cdot 7$, which is a contradiction.
\end{proof}

\section{Biplanes with parameters $(79,13,2)$}\label{sec:79}

As mentioned in Section~\ref{sec:examples},  the  only known biplanes with parameters  $(79,13,2)$ are the Aschbacher biplane and its dual. Moreover if $\Dmc$ is any other biplane with parameters $(79, 13, 2)$, then it follows from \cite[Theorem 2]{a:Marangunic-92} that the full automorphism group of $\Dmc$ is a $3$-group. We seek to restrict $\Aut(\Dmc)$ further, to having order either 1 or 3.

\begin{theorem}\label{thm:79}
If $\Dmc$ is a biplane with parameters $(79,13,2)$, then either $\Dmc$ or its dual is the Aschbacher biplane with $|\Aut(\Dmc)|=110$,
or  $\Aut(\Dmc)\leq \C_{3}$.
\end{theorem}

\begin{proof}
Let $G=\Aut(\Dmc)$ and suppose that neither $\Dmc$ nor its dual is the Aschbacher biplane, so that, by
\cite[Theorem 2]{a:Marangunic-92}, $G$ is a $3$-group. Suppose that $1\ne x\in G$.  Since $o(x)$ is a $3$-power, it follows that
$f=|\fix_{\Pmc}(x)|\equiv 1\pmod{3}$, so $f>0$ and hence also $f= |\fix_{\Bmc}(x)|>0$ by Lemma~\ref{lem:fix}. Let $\alpha\in\fix_\Pmc(x)$, $B\in\fix_\Bmc(x)$,  and $s:=\sbf_{B} = |\fix_{[B]}(x)|$. Then $s\equiv k\equiv 1 \mod 3$, so also $s>0$. By Lemma~\ref{lem:fix-bi}(d) (noting that $\rbf_B=0$ since $o(x)$ is odd and that $k-2=11$ is not a square), it follows that $s(s-1)\leq 11$; and then by
Lemma~\ref{lem:fix-bi}(b), $f=s(s-1)/2 +1$, so $f\leq 6$. Since $f\equiv s\equiv 1\pmod{3}$, it follows that $f=s=1$.
Therefore, every nontrivial element $x$ of $G$ fixes exactly one point $\alpha$ and exactly one block $B$, and $\alpha\in B$.

Now the group $G$ is a nontrivial $3$-group and as $v=79$, it follows that $G$ fixes some point, say $\alpha\in\Pmc$. Thus $G=G_\alpha$ and $\alpha$ is the unique fixed point of each nontrivial element of $G$. Since $|\Pmc\setminus\{\alpha\}|=78\equiv 6\pmod{9}$ it follows that $G$ has an orbit $\Delta$ of length  $3$ in $\Pmc$. Let $\beta\in\Delta$. Then $G_\beta$ fixes $\{\alpha\}\cup\Delta$ pointwise, and  hence by the previous paragraph  $G_\beta=1$. Thus $|G|=3|G_\beta|=3$.
\end{proof}

\section{Biplanes with parameters $(121,16,2)$}\label{sec:121}

Let $\Dmc=(\Pmc, \Bmc)$ be a biplane with parameters $(121,16,2)$ and let $G=\Aut(\Dmc)$. We use the notation and results from  Section~\ref{sec:fix}. In this section, we prove Theorem~\ref{thm:main-121} for $k=121$.

\subsection{$2$-elements in $G$}
Here we bound the $2$-part of $|\Aut(\Dmc)|$.

\begin{proposition}\label{prop:2bound}
A Sylow $2$-subgroup of $\Aut(\Dmc)$ has order at most $2^7$.
\end{proposition}

First we restrict in Lemmas~\ref{lem:2-1}, \ref{lem:2-2} and \ref{lem:2-3} the possible cycle types of $2$-elements.

\begin{lemma}\label{lem:2-1}
Let $x\in\Aut(\Dmc)$ with $o(x)=2$, let $f:=|\fix_\Pmc(x)|$, and let $\sbf_\alpha=|\fix_{[\alpha]}(x)|$ for each $\alpha\in\fix_\Pmc(x)$. Then
\begin{enumerate}[{\rm (a)}]
\item $f>0$ and either $(f, \sbf_\alpha)=(13,4)$ for each $\alpha\in\fix_\Pmc(x)$, or $f=9$ and $\sbf_\alpha\in\{0,2\}$  for each $\alpha\in\fix_\Pmc(x)$;
\item $f =|\fix_\Bmc(x)|$ and setting $\sbf_B = |\fix_{[B]}(x)|$ for $B\in\fix_\Bmc(x)$, either $f=13$ and $\sbf_B=4$ for each $B\in\fix_\Bmc(x)$, or $f=9$ and $\sbf_B\in\{0,2\}$  for each $B\in\fix_\Bmc(x)$.
\end{enumerate}
\end{lemma}
\begin{proof}
Since $v=121$ is odd, $f>0$. Let $\alpha\in\fix_\Pmc(x)$. Then by Lemma~\ref{lem:bi-2}(a),  $f=(k+1+(\sbf_{\alpha}-1)^{2})/2$, and either  $\sbf_\alpha$ is independent of $\alpha\in\fix_\Pmc(x)$, or each $\sbf_\alpha\in\{0,2\}$. Since $k-2=14$ is not a square, it follows from Lemma~\ref{lem:bi-2}(c) that $(\sbf_\alpha-1)^2\leq k-5=11$, and since $f$ is an integer, this implies that either $(f,\sbf_\alpha)=(13, 4)$, or $\sbf_\alpha\in\{0,2\}$ and $f=9$. This proves part (a).

By  Lemma~\ref{lem:fix}, $f=|\fix_\Bmc(x)|$. Let $B\in\fix_\Bmc(x)$. Then by Lemma~\ref{lem:fix-bi}(b), $f=\sbf_B(\sbf_B-1)/2+\rbf_B+1$, where   $\rbf_B$ is the number of $2$-cycles of $x$ in $[B]$. Note that $\rbf_B\leq k/2=8$, so $\sbf_B$ is even since $k$ is even.
If $\sbf_B=0$, then $f=\rbf_B+1\leq 9$, so by part (a), $f=9$ and $\rbf_B=8$.

\emph{We claim that, if $\sbf_B>0$ and $\alpha\in\fix_{[B]}(x)$, then  $\sbf_\alpha\geq \sbf_{B}$.} Under these conditions,
for each $\beta\in\fix_{[B]}(x)\setminus\{\alpha\}$, the second block containing $\{\alpha,\beta\}$ is also fixed by $x$. These blocks together with $B$ give $\sbf_{B}$ blocks in $\fix_{[\alpha]}(x)$. Thus $\sbf_\alpha\geq \sbf_{B}$, proving the claim.

Suppose first that $f=13$, so by part (a),  $\sbf_\alpha=4$ for all $\alpha\in\fix_\Pmc(x)$. Also  $\sbf_B(\sbf_B-1)/2=f-1-\rbf_B=12-\rbf_{B}\geq 4$ so $\sbf_{B}\geq 4$. However, by the claim we have $4=\sbf_\alpha\geq \sbf_{B}$, and hence $\sbf_{B}=4$ for all $B\in\fix_\Bmc(x)$, as in part (b).
Suppose now that $f=9$, so by part (a), $\sbf_\alpha\in\{0,2\}$  for each $\alpha\in\fix_\Pmc(x)$. If $\sbf_B>0$ and $\alpha\in\fix_{[B]}(x)$,  then by the claim,  $\sbf_B\leq\sbf_\alpha\leq 2$. Hence $\sbf_B=\sbf_\alpha= 2$ (as $\sbf_B$ is even). Thus   $\sbf_B\in\{0,2\}$  for each $B\in\fix_\Bmc(x)$, and (b) is proved.
\end{proof}

\begin{lemma}\label{lem:2-2}
Let $x\in\Aut(\Dmc)$ with $o(x)=2^a\geq4$ and  $f:=|\fix_\Pmc(x)|>1$.
\begin{enumerate}[\rm (a)]
\item Then $o(x)=4$, and $f =|\fix_\Bmc(x)|$.
\item Let $B\in\fix_{\Bmc}(x)$. Then $|\fix_\Pmc(x^2)|=13$ and $|\fix_{[B]}(x^2)|=4$.
\item Let $\sbf_{B}=|\fix_{[B]}(x)|$ and $\rbf_{B}$ be the number of $2$-cycles of $x$ in $B$.
Then one of the following holds.
\begin{enumerate}[\rm (i)]
\item $x$ has cycle type $1^3\, 2^5\, 4^{27}$, and either $(\sbf_{B}, \rbf_{B})= (0,2)$  for all $B\in\fix_{\Bmc}(x)$, or  $(\sbf_{B}, \rbf_{B})=(2,1)$ for two of the fixed blocks  of $x$ and $(0,2)$ for the third.
\item $x$ has cycle type $1^7\, 2^3\, 4^{27}$, and $(\sbf_{B}, \rbf_{B})=(4,0)$ for all $B\in\fix_{\Bmc}(x)$.
\end{enumerate}
\end{enumerate}
\end{lemma}

\begin{proof}
By Lemma~\ref{lem:fix}, $f=|\fix_\Bmc(x)|>1$. Choose $B\in\fix_\Bmc(x)$. Then by Lemma~\ref{lem:fix-bi}(b), $f=1+\sbf_B(\sbf_B-1)/2+\rbf_{B}$,
and since $f>1$, the pair $(\sbf_{B}, \rbf_{B})\ne (0,0)$ or $(1,0)$. Let $y=x^{2^{a-1}}$, the involution in $\langle x\rangle$, and let $f_y=|\fix_\Pmc(y)|$.

Suppose first that $\sbf_{B}>0$. Then since $k=16$, $\sbf_{B}$ is even so $\sbf_{B}\geq2$. Let $\alpha\in \fix_{[B]}(x)$. Then, for each of the other $\sbf_{B}-1$ points $\alpha'\in\fix_{[B]}(x)\setminus\{\alpha\}$, the second block $B'$ containing $\{\alpha,\alpha'\}$ is also fixed by $x$, and these blocks are distinct for distinct $\alpha'$ (since $\Dmc$ is a biplane). These blocks, together with $B$, give $\sbf_{B}$ blocks on $\alpha$ which are fixed by $x$, and hence are also fixed  by $y$. Also, by Lemma~\ref{lem:fix-bi}(a), there are $\rbf_\alpha=\rbf_{B}$ cycles of $x$ of length $2$ on the set $[\alpha]$ of blocks containing $\alpha$. Each of the $2\rbf_{B}$ blocks in these $2$-cycles of $x$ is fixed by $y$, and hence $s_y:=|\fix_{[\alpha]}(y)|\geq \sbf_{B}+2\rbf_{B}\geq2$.

Suppose that $\sbf_{B}\geq4$. Then $s_y\geq 4$, and it follows from Lemma~\ref{lem:2-1} applied to $y$ that $(f_y,s_y)=(13,4)$. Thus $\sbf_{B}=4, \rbf_{B}=0$, and $f=1+\binom{4}{2}=7$. This means that each $x$-cycle in $B\setminus \fix_{[B]}(x)$ has size a multiple of $4$, and since $|B\setminus\fix_{[B]}(x)|=12\not\equiv 0\pmod{8}$, it follows that $x$ must have at least one $4$-cycle in $B$, say $(\beta_1,\dots,\beta_4)$. If $B_i$ is the second block containing $\{\alpha, \beta_i\}$, for
 $1\leq i\leq 4$, then  the blocks $B_1,\dots,B_4$ are pairwise distinct (since distinct blocks intersect in exactly two points), each lies in $[\alpha]$, and they are permuted cyclically by $\langle x\rangle$.     If $o(x)\geq8$, then  $y$ would also fix each of the $B_i$, implying that $s_y>4$, which is a contradiction by Lemma~\ref{lem:2-1}. Hence $o(x)=4$ and $x$ has $(f_y-f)/2=(13-7)/2=3$ cycles of length $2$, and hence has cycle type $1^7\, 2^3\, 4^{27}$. Thus all assertions are proved, except for showing that $\sbf_{B}$ is independent of the choice of $B$. We defer this.

 Assume next that $\sbf_{B}=2$. Then in its action on the 14 points of $B\setminus\fix_{[B]}(x)$, $x$ must have at least one $2$-cycle, so $\rbf_{B}\geq 1$.  Then $s_y\geq \sbf_{B}+2\rbf_{B}\geq4$, and again by Lemma~\ref{lem:2-1}, we obtain $(f_y,s_y)=(13,4)$, so $(\sbf_{B}, \rbf_{B})=(2,1)$.
 Hence $f=1+\binom{2}{2}+1=3$, and since there are 12 points of $B$ lying in  $x$-cycles of length at least 4, $x$ must have at least one $4$-cycle in $B$. If $o(x)\geq 8$, then the hypotheses of the lemma hold also for $x^2$, and $x^2$ has two $2$-cycles in the $x$-cycle of length 4 in $B$ and also $|\fix_{[B]}(x^2)|=4$. Since  $y$ is also the involution in $\langle x^2\rangle$, the argument just given implies that $s_y\geq 4+2\times 2=8$, which is a contradiction  by Lemma~\ref{lem:2-1}. Thus $o(x)=4$ in this case also and $x$ has $(f_y-f)/2=(13-3)/2=5$ cycles of length $2$, and  hence has cycle type $1^3\, 2^5\, 4^{27}$.  Again   all assertions are proved in this case, except for dealing with the values of $\sbf_{B}$ for other fixed blocks $B$. We defer this until we have considered the case $\sbf_{B}=0$.

Now suppose that $\sbf_{B}=0$. Then $\rbf_{B}>0$ since here $f=1+\rbf_{B}>1$.
The number $2\rbf_{B}$ of points of $B$ lying in $2$-cycles of $x$ must be a multiple of $4$, so $\rbf_{B}$ is even,
and hence $\rbf_{B}\geq2$. This means that $y\in\langle x^2\rangle$ fixes at least $2\rbf_{B}\geq 4$ points of $B$, and hence by Lemma~\ref{lem:2-1}, $(f_y,s_y)=(13,4)$. It follows that $\rbf_{B}=2$, and hence that $f=1+\rbf_{B}=3$, and $x$ has at least one $4$-cycle in $B\setminus \fix_{[B]}(y)$. Now $y$ cannot act trivially on the points of this $4$-cycle (since $s_y=4$), and hence $y\not\in\langle x^4\rangle$. Thus $o(x)=4$, $y=x^2$, and $x$ has $(f_y-f)/2 = (13-3)/2=5$ cycles of length $2$ in $B$,  and  hence has cycle type $1^3\, 2^5\, 4^{27}$.

Finally we consider blocks $B'\in\fix_\Bmc(x)\setminus\{B\}$. If $\sbf_{B'}\in\{0,2\}$, then our arguments above show that $x$ has $f=3$ fixed points. Thus in  case (c)(ii), where $f=7$, we must have $(\sbf_{B'}, \rbf_{B'})=(4,0)$ for all $B'$. Assume now that $x$ has $f=3$ fixed points so we are in case (c)(i), and $x$ has three fixed blocks, say $B, B', B''$. If $x$ has no fixed points in any of these blocks there is nothing to prove, so suppose without loss of generality that  $(\sbf_{B}, \rbf_{B})=(2,1)$, and let $\fix_{[B]}(x)=\{\alpha,\beta\}$. Then the second block containing $\{\alpha,\beta\}$ is also fixed by $x$; suppose that it is $B'$, so  $(\sbf_{B'}, \rbf_{B'})=(2,1)$. The third fixed block $B''$ intersects
$B$ in a pair of points that must be left invariant by $x$ and is disjoint from $\{\alpha,\beta\}$ (since $\Dmc$ is a biplane). Thus $B''\cap B$ is a $2$-cycle of $x$, and similarly $B''\cap B'$ is a $2$-cycle of $x$. Since $B$ and $B''$ are the only blocks containing $B''\cap B$, these $2$-cycles of $x$ in $B''$ are distinct, and it follows that $\rbf_{B''}=2$ and therefore (from the argument above) that $\sbf_{B''}=0$. This completes the proof.
\end{proof}


\begin{lemma}\label{lem:2-3}
Let $x\in\Aut(\Dmc)$ with $o(x)=2^a\geq4$ and suppose that $f:=|\fix_\Pmc(x)|\leq 1$. Then $a\in\{2,3\}$, $f=1$, say    $\fix_{\Bmc}(x)
=\{B\}$, and all $x$-cycles in $B$ have lengths at least  $4$. Moreover  {\rm Table~\ref{T1}} gives the possible values for $a$,
the cycle type of $x$, the numbers $r_4, r_8$  of $x$-cycles in $B$ of lengths $4$ and $8$, respectively, and information about the element $x^2$ (where, if $o(x)=4$, then $f_{x^2}, s_{x^2}$ are the parameters $f, \sbf_\alpha$ in Lemma~$\ref{lem:2-1}$ applied to $x^2$).
\end{lemma}
\begin{table}[h]
    \centering
    \small
    \caption{Cycle types of $2$-elements for Lemma~\ref{lem:2-3}}\label{T1}
    \begin{tabular}{lcclll}
        \hline
        $a$  & $o(x)$ & Cycle type of $x$   & $r_4$ & $r_8$  & Description of $x^2$\\
        \hline
        $2$ & $ 4$      & $1^1\, 2^4\, 4^{28}$ & $4$  & $0$  & $x^2$ as in Lemma~\ref{lem:2-1} with $(f_{x^2}, s_{x^2})=(9, 0)$\\
        $3$ & $ 8$      & $1^1\, 4^2\, 8^{14}$ & $0$  & $2$  & $x^2$ as in line for $a=2$\\
                \hline
    \end{tabular}
\end{table}

\begin{proof}
Since $v=121$ is odd, we must have $f>0$ and hence $f=1$ since we are assuming $f\leq 1$. By Lemma~\ref{lem:fix}, $x$ also fixes exactly one block, say $B$. By Lemma~\ref{lem:fix-bi}(b), $f=1+\sbf_{B}(\sbf_{B}-1)/2+\rbf_{B}$, where $\sbf_{B}=|\fix_{[B]}(x)|$ and $\rbf_{B}$ is the number of $2$-cycles of $x$ in $B$. Hence $\rbf_{B}=0$ and $\sbf_{B}\leq 1$. Since $|B|=16$, $\sbf_{B}$ is even and so $\sbf_{B}=0$. Thus all $x$-cycles in $B$ have lengths greater than $2$.

Suppose first that the number $r_4$ of $4$-cycles in $B$ is at least 1. Since $|B|=16$ it follows that $r_4$ is even, so $r_4\geq 2$.  Then $x^2$ has  no fixed points in $B$ and has $2r_4\geq 4$ cycles in $B$ of length 2, and so by  Lemma~\ref{lem:fix-bi}(b) applied to $x^2$,  $x^2$ fixes $1+2r_4\geq 5$ blocks. By Lemma~\ref{lem:fix} again, $x^2$ also fixes $1+2r_4$ points. Suppose in addition that $a\geq3$ so that $o(x^2)\geq4$. Applying Lemma~\ref{lem:2-2} to $x^2$ gives a contradiction, as (c)(i) cannot hold since $x^2$ fixes more than $3$ points, and (c)(ii) cannot hold since $x^2$ has no fixed points in $B$. Thus we conclude that $a=2$ in this case.  This means that $r_4=4$, and hence $x^2$ fixes $B$ setwise and has eight $2$-cycles in $B$. By Lemma~\ref{lem:2-1}, $x^2$ has the parameters $(f_{x^2}, s_{x^2})=(9, 0)$. Hence $x$ has $(f_{x^2}-f)/2=(9-1)/2=4$ cycles of length $2$, and hence the line for $a=2$ of Table~\ref{T1} holds.

Thus we may assume that $r_4=0$. Suppose next that $r_8>0$. Then $o(x)\geq 8$ and  $x$ has $r_8=2$ cycles of length
$8$ in $B$. This means that $x^2$ has four $4$-cycles in $B$. If $x^2$ fixes more than 1 point, then  Lemma~\ref{lem:2-2} applies and leads to a contradiction since such elements have no fixed blocks on which they act with four $4$-cycles. Thus $x^2$ has exactly 1 fixed point and satisfies the hypotheses of this Lemma~\ref{lem:2-3}. It follows from the previous paragraph that $o(x^2)=4$ and the properties of the `$a=2$' line of Table ~\ref{T1} hold for $x^2$.   We deduce from this that $o(x)=8$ and $x$ has cycle type   $1^1\, 4^2\, 8^{14}$, so the properties of the line for $a=3$ of Table~\ref{T1} hold for $x$.

Finally assume that $r_4=r_8=0$. Then since $x$ fixes $B$ setwise it must act on $B$ as a $16$-cycle. In particular $o(x)\geq 16$, and $x^2$ has two $8$-cycles in $B$. If $x^2$ has more than one fixed point, then Lemma~\ref{lem:2-2} applies to $x^2$ and we have a contradiction since $o(x^2)\geq 8$. Thus $x^2$ has just one fixed point and satisfies the hypotheses of Lemma~\ref{lem:2-3}. It now follows from the previous paragraph that $o(x^2)=8$  and   the properties of the  `$a=3$' line of Table ~\ref{T1} hold for $x^2$.   We deduce from this that $o(x)=16$ and $x$ has cycle type   $1^1\, 8^1\, 16^{7}$.
Let $y=x^8$. Then from Table~\ref{T1} it follows that the involution  $y$ has the parameters $(f_{y}, s_{y})=(9, 0)$ from Lemma~\ref{lem:2-1} and $y$ fixes nine points and nine blocks, and the fixed points of $y$ are precisely the nine points lying in $x$-cycles of length 1 or 8. Also as argued in the second paragraph of this proof, $y$ has eight $2$-cycles in $B$,
and by Lemma~\ref{lem:2-1}, $y$ has eight $2$-cycles in each of its fixed blocks.
For each of the $2$-cycles $(\alpha_i,\beta_i)$ of $y$ in $B$, $y$ fixes the second block $B_i$ containing
$\{\alpha_i,\beta_i\}$. The nine blocks $B, B_1,\dots, B_8$ are pairwise distinct (since $\Dmc$ is a biplane),
and hence are precisely the fixed blocks of $y$. Since $x$ centralises $y$ and fixes $B$, it
follows that $x$ leaves invariant $\cup_{i=1}^8 B_i$. Now for each  $B_i$, and each $j\ne i$, $B_i\cap B_j$ has size $2$, and these $2$-subsets are distinct $2$-cycles of $y$ for distinct $j$. It follows that each of the seven $y$-cycles of length
$2$ in $B_i\setminus\{\alpha_i,\beta_i\}$ occurs as $B_i\cap B_j$ for a unique $j\ne i$.
Thus the number of pairs $(B_i,\gamma)$, where $i=1,\dots, 8$ and $\gamma\in B_i\setminus\{\alpha_i,\beta_i\}$
can be counted in two ways and this gives $|\cup_{i=1}^8 B_i|\cdot 2 = 8\cdot 14$, so
$|\cup_{i=1}^8 B_i| = 56$. This is a contradiction, since $\cup_{i=1}^8 B_i$ is an $x$-invariant
set of size $56$ and is contained in a union of $x$-cycles, each of length 16.
\end{proof}


\subsection*{\textbf Proof of Proposition~\ref{prop:2bound}} Let $P$ be a Sylow $2$-subgroup of $\Aut(\Dmc)$ and suppose that $P\ne1$. Since $v=121$ is odd, $P$ leaves some block invariant, say $B$. Any nontrivial element of $P$ fixing $B$ pointwise would have at least $16$ fixed points, and it follows from Lemmas~\ref{lem:2-1} and \ref{lem:2-2} that there are no such elements. Hence $P$ acts faithfully on $B$, that is, $P\cong P^{[B]}\leq\Sym([B])$. By  Lemmas~\ref{lem:2-1}, \ref{lem:2-2} and \ref{lem:2-3},  each nontrivial element of $P$ fixes at most four points of $[B]$. Suppose that the shortest nontrivial $P$-orbit  in $[B]$ (that is, of size greater than 1) has length $2^a$ and let $\alpha$ be a point of this orbit. Then $P_\alpha$ fixes at least 2 points of $[B]$. If $P_\alpha=1$, then  $|P|=2^a\leq [B]=16$, so suppose that $P_\alpha\ne 1$. Suppose that the shortest nontrivial $P_\alpha$-orbit in $[B]$ has length $2^b$ and let $\beta$ be a point of this orbit. Note that $b\leq a$ and if $b=a$, then  the $P_\alpha$-orbit containing $\beta$ must be equal to the $P$-orbit containing $\beta$, and hence $P$ must have two orbits (at least) in $[B]$ of length $2^a$ so $2^a\leq 8$ in this case.  Hence we always have $2^{a+b}\leq 2^7$. Also $P_{\alpha,\beta}$ fixes at least 4 points of $[B]$. If $P_{\alpha,\beta}=1$, then  $|P|=2^{a+b}\leq 2^7$, so suppose that $P_{\alpha,\beta}\ne 1$. Then $P_{\alpha,\beta}$ fixes exactly 4 points of $[B]$. This implies that $P_\alpha$ fixes exactly 2 points of $[B]$ and hence must have an orbit of length $2$ in $[B]$, so $2^b=2$. Further, $P_{\alpha,\beta}$ must act semiregularly on the 12 remaining points of $[B]$ as a group of order $2$ or $4$ (by Lemmas~\ref{lem:2-1} and \ref{lem:2-2}). Thus $|P|\leq 2^{a+1}\cdot 4\leq 2^7$.

\subsection{Odd order $p$-subgroups of automorphisms}\label{sec:121-sp}

In this subsection, we study the $p$-subgroups of $\Aut(\Dmc)$ for odd primes $p$..

\begin{lemma}\label{lem:121-s3}
Let $\Dmc$ be a biplane with parameters $(121,16,2)$. Then a Sylow $3$-subgroup of $\Aut(\Dmc)$ is elementary abelian of order at most $3^{2}$.
\end{lemma}

\begin{proof}
Suppose that $x\in\Aut(\Dmc)$ with $o(x)=3^{a}\geq 3$ fixing $f$ points. We claim that $f\in \{1,7\}$. Since $3$ does not divide $v=121$, $x$ fixes at least one point, and so by Lemma \ref{lem:fix}, $x$ has to fix a block, say $B$. Let $s:=\sbf_{B}=|\fix_{[B]}(x)|$. Since also $3$ does not divide $k=16$, it follows that $s>0$. Note that $k-2=16-2=14$ is not a square, so Lemma~\ref{lem:fix-bi}(d) implies that $s^{2}-s\leq 14$, and since  $s \equiv 16\equiv 1 \mod 3$, we must have $s\in\{1,4\}$, and since $f=s(s-1)/2+1$ by Lemma~\ref{lem:fix-bi}(b), we conclude that $f$ is $1$ or $7$, proving the claim.

Now let $P$ be a Sylow $3$-subgroup of $\Aut(\Dmc)$ and assume that $P\ne1$. Since $\gcd(v,3)=1$, $P$ fixes setwise at least one block of $\Dmc$, say $P$ fixes $B$. We have just shown that all nontrivial $3$-elements fix at most $7$ points, and hence $P$ acts faithfully on $[B]$.   Moreover, we showed that each nontrivial $3$-element fixes either $1$ or $4$ points of $B$.  Since $\gcd(|B|,3)=1$ this implies that $P$ fixes either 1 or 4 points of $[B]$. In either case $P$ must have at least one orbit $\Delta$ in $[B]$ of length 3 (since neither $16-1$ nor $16-4$ is divisible by $9$). Let $\alpha\in\Delta$. Then $|P:P_\alpha|=3$ and $P_\alpha$ fixes at least $4$ points of $B$. If $P_\alpha=1$, then  $|P|=3$, so suppose that $P_\alpha\ne 1$. Since no nontrivial element of $P_\alpha$ fixes more than $4$ points of $B$, it follows that $|\fix_{[B]}(P_\alpha)|=4$ and hence $P_\alpha$ has at least one orbit in $B$ of length $3$. For $\beta$ in this orbit, $P_{\alpha,\beta}$ fixes more than $4$ points of $B$ and hence must be trivial. Thus in this case, $|P|=9$, and $P$ has at least two orbits of length 3 in $[B]$, say $\Delta$ and $\Delta'$. We have shown moreover that the stabiliser of a point of $\Delta'$ is still transitive on $\Delta$, and hence that $P$ acts faithfully on $\Delta'\cup\Delta$. It follows that $P\cong\C_3^2$.
\end{proof}

\begin{lemma}\label{lem:121-sp}
Let $\Dmc=(\Pmc,\Bmc)$ be a biplane with parameters $(121,16,2)$. Then a Sylow $p$-subgroup of $\Aut(\Dmc)$ with $p\in\{5,7,11,13\}$ is cyclic of order at most $p$. Moreover, the cycle type of a $p$-element $x\in \Aut(\Dmc)$ on $\Pmc$ and on a  fixed block (if such exists) is given in {\rm Table~\ref{tbl:121-sp}}.
\end{lemma}
\begin{table}[h]
  \centering
  \caption{Cycle type of $p$-elements $x$ with $p\in\{5,7,11,13\}$ in Lemma~\ref{lem:121-sp}.}\label{tbl:121-sp}
  \begin{tabular}{lllll}
    \hline
    $p$ & $s$ & $f$ & Cycle type on $\Pmc$ & Cycle type on a fixed block $B$ \\
    \hline
    $5$& $1$ & $1$ & $1^{1}5^{24}$ & $1^{1}5^{3}$\\
    $7$& $2$&$2$ & $1^{2}7^{17}$& $1^{2}7^{2}$\\
    $11$& $0$ & $0$& $11^{11}$ & $x$ fixes no blocks  \\
    $13$& $3$& $4$ & $1^{4}13^{9}$& $1^{3}13^{1}$\\
    \hline
  \end{tabular}
\end{table}

\begin{proof}
Suppose that $x\in\Aut(\Dmc)$ with $o(x)=p^{a}>p$ fixing $f$ points with $p\in\{5,7,11,13\}$. We use the same argument as in Lemma~\ref{lem:121-s3}.

Assume first that $p\neq 11$. Then since $p$ is coprime to $121$, $x$ fixes at least one point, and so by Lemma \ref{lem:fix}, $x$ fixes also a block, say $B$. Since $p$ does not divide $k=16$ and $k-2=16-2=14$ is not a square, it follows from Lemma~\ref{lem:fix-bi}(d) that
$s^{2}-s\leq 14$, where $s:=\sbf_{B}=|\fix_{[B]}(x)|$. Since $s \equiv 16\mod p$, we must have
\begin{align*}
  \begin{array}{ll}
    s\in\{1,6,11,16\}, & \hbox{if $p=5$;} \\
    s\in\{2,9\}, & \hbox{if $p=7$;} \\
    s\in\{3\}, & \hbox{if $p=13$.}
  \end{array}
\end{align*}
By  Corollary~\ref{cor:fix-bi-bound},  $f=s(s-1)/2+1$ and $f\leq k/2=8$, and we obtain $s$ and $f$ as in the  second and third columns of Table~\ref{tbl:121-sp}, and hence every $p$-element has exactly $s$ fixed points in $B$ where $s$ is listed in the second column  Table~\ref{tbl:121-sp}..

Now let $P$ be a Sylow $p$-subgroup of $\Aut(\Dmc)$ where $p\in\{5,7,13\}$. Then $P$ acts faithfully on $B$. Since $p^{2}$ does not divide $k-s$, it follows that $P$ has an orbit $\Delta$ of length $p$ in $[B]$. Let $\alpha\in \Delta$. If $P_{\alpha}\neq 1$, then $\Aut(\Dmc)$ would have a $p$-element fixing more than $s$ points in $B$, which is a contradiction. Thus $P_{\alpha}=1$, and hence $|P|=|P:P_{\alpha}|=|\Delta|=p$. Therefore, $P$ is cyclic of order $p$ and the  cycle types of the non-identity elements of $P$ on $\Pmc$ and on $B$ are as in the fourth and fifth columns of Table \ref{tbl:121-sp}, respectively.

Assume now that $p=11$. By Proposition~\ref{cor:p2}, each element of order $p$ has no fixed points, and hence $x$ has   $f=0$ fixed points. Let $P$ be a Sylow $11$-subgroup of $\Aut(\Dmc)$. Then $P$ acts semiregularly on $\Pmc$. If  $|P|=11^{2}$, then  $\Dmc$ is transitive which contradicts Corollary~\ref{cor:121-t}. Thus, $|P|=11$,  and every $11$-element has cycle type $11^{11}$, as in Table~\ref{tbl:121-sp}.
\end{proof}

\subsection{Proof of Theorem~\ref{thm:main-121}}\label{sec:proof-thm-121}
Part (a) follows from Theorem~\ref{thm:79}. To prove part (b), 
let $\Dmc=(\Pmc, \Bmc)$ be a biplane with parameters $(121,16,2)$. Then by Lemma~\ref{lem:Autprims}, the only primes that could possibly divide $|\Aut(\Dmc)|$ are $2$, $3$, $5$, $7$, $11$ and $13$. The divisibility conditions now follow from Proposition~\ref{prop:2bound} and Lemmas~\ref{lem:121-s3} and~\ref{lem:121-sp}.

\section{Biplanes and product action}\label{sec:ProdAction}

As mentioned in the introduction, for each of the primitive types in Table~\ref{tbl:prim} except $\HS$ and $\SD$, there is some primitive group of that type which preserves a homogeneous cartesian decomposition.
In this section, we study biplanes $\Dmc=(\Pmc,\Bmc)$ which admit a point-primitive group $G$ of automorphisms
preserving such a decomposition $\Pmc=\Gamma^d$ for some $d\geq2$. Then $G$ is permutationally isomorphic to a subgroup of a wreath product $W=H\wr \S_d$ in its product action on $\Gamma^d$, for some primitive subgroup $H\leq \Sym(\Gamma)$  such that
$G\cap (H\times (H\wr \S_{d-1}))$ maps onto $H$ under the first coordinate projection  (\cite[Theorems 5.13 and 5.18]{b:Praeger-18}).  Moreover, if $G$ has type $\HC, \CD$ or $\PA$, then $G$ contains $\Soc(H)^d$, where $\Soc(H)$ is the socle of $H$
(\cite[Theorems 7.7 and 11.13]{b:Praeger-18}). Thus in these cases $G\cap H^d$ contains $\Soc(H)\times 1^{d-1}$, which in particular is not semiregular on $\Pmc$.

We make the following hypotheses which establish our notation, and we note that the group $L$ defined there is not semiregular
for groups $G$ of type $\HC, \CD$ or $\PA$ (that is, some non-identity element of $L$ fixes a point of $\Pmc$). Our first result Lemma \ref{lem:main} can be viewed as a reduction result when applied to groups of these types.

\begin{hypothesis}\label{hypo}
Let $\Dmc=(\Pmc,\Bmc)$ be a biplane with parameters $(v, k, 2)$ and $k\geq4$, and let $G\leq \Aut(\Dmc)$ be point-primitive
preserving a cartesian decomposition  $\Pmc = \Gamma^d$, with $d\geq 2$ and $|\Gamma|=c\geq 5$, so that
$G\leq W$, where $W=H \wr \S_d$ with $H\leq \Sym(\Gamma)$ primitive. Let $\sigma: H\times (H\wr \S_{d-1}) \to H$ denote the projection to the first coordinate, and assume that $\sigma$ maps $G\cap (H\times (H\wr \S_{d-1}))$ onto $H$. Let $L = G \cap (H \times 1^{d-1})$,
so $L\cong \sigma(L)\unlhd H$.
\end{hypothesis}

\begin{lemma}\label{lem:main}
Suppose that {\rm Hypothesis \ref{hypo}} holds.  Then either $L$ is semiregular (on $\Pmc$ and on $\Gamma$), or $d=2$ and $k=\lceil\sqrt{2}\,c\rceil$ with $c\geq 11$ and $c\ne 28$.
\end{lemma}

\begin{proof}
Since elements of $L$ move only the first entries of points of $\Pmc=\Gamma^d$, $L$ is semiregular on $\Pmc$ if and only if $\sigma(L)$ is semiregular on $\Gamma$. Suppose that this is not the case. Then for some $1\ne x\in L$, $\sigma(x)$ fixes a point $\alpha\in \Gamma$.
Now $x$ fixes at least $c^{d-1}$ points in $\Pmc$, namely those with $\alpha$ as their first entry. So  $|\fix_{\Pmc}(x)|\geq c^{d-1}$. By Corollary~\ref{cor:fix-bi-bound},  $|\fix_{\Pmc}(x)|\leq k+\sqrt{k-2}<2k-2$. Therefore,
\begin{align}\label{eq:lem-PA-as-k-1}
    k> \frac{c^{d-1}+2}{2}.
\end{align}
On the other hand, by Lemma~\ref{lem:basic},  $k(k-1)=2(v-1)$, that is to say,
\begin{align}\label{eq:lem-PA-as-k-2}
    k(k-1)=2c^{d}-2,
\end{align}
and so \eqref{eq:lem-PA-as-k-1} implies that
\[
    8c^d > 8(c^{d}-1)=4k(k-1)> (c^{d-1}+2)(c^{d-1}+2)> c^{2d-2}.
\]
Since $c\geq 5$, this implies that either (i) $d=2$, or (ii) $d=3$ and $c=5,6,7$. In case (ii), there is no integer $k$ satisfying \eqref{eq:lem-PA-as-k-2}. Thus $d=2$, and Lemma \ref{lem:basic} implies that $k=\lceil\sqrt{2}\,c\rceil$. It is easy to check that, for $c\leq 10$ and $c=28$, no integer $k$ satisfies \eqref{eq:lem-PA-as-k-2}, and the proof is complete.
\end{proof}


By Lemma~\ref{lem:main}, if {\rm Hypothesis \ref{hypo}} holds and $L$ is not semiregular, then $d=2$ so $\Pmc:=\Gamma^{2}$ and  $v=c^{2}$. We introduce more notation for the case $d=2$.
For $i=1,2$, let $\pi_{i}$ be the projection map $\Gamma^{2}\to \Gamma$ such that $\pi_1:(\alpha,\beta)\rightarrow \alpha$ and $\pi_2:(\alpha,\beta)\rightarrow \beta$, and set
\begin{align}\label{eq:d2-s1s2}
P_{1}&=\{\{p_{1},p_{2}\}\mid p_{1}, p_2 \in\Gamma^2, p_1\ne p_2, \text{ and } \pi_{1}(p_{1})=\pi_{1}(p_{2})\};\\
P_{2}&=\{\{p_{1},p_{2}\}\mid p_{1}, p_2 \in\Gamma^2, p_1\ne p_2,  \text{ and } \pi_{2}(p_{1})=\pi_{2}(p_{2})\}.
\end{align}
Note that $P_{1}\cap P_{2}=\emptyset$. Our next result looks at the case $d=2$ (whether or not $L$ is semiregular).

\begin{lemma}\label{lem:threepoints}
Suppose that {\rm Hypothesis \ref{hypo}} holds with  $d=2$, and let $B\in\Bmc$. Then
\begin{enumerate}[\rm (a)]
\item  exactly $2(c-1)$ of the unordered pairs of points of $B$ lie in $P_1\cup P_2$;
    \item  there is a $3$-element subset $S$ of $B$, such that all pairs from $S$ lie in $P_1$ or all pairs from $S$ lie in $P_2$;
    \item  for all pairwise distinct points $\alpha,\beta,\gamma\in\Gamma$,  $\sigma(L)_{\alpha,\beta,\gamma}=1$.
\end{enumerate}
\end{lemma}

\begin{proof}
(a) The group $G$ leaves $P_1\cup P_2$ invariant.
Since $G$ is transitive on $\Bmc$, the number of $m$ of  pairs from $B$ that lie in $P_1\cup P_2$  is independent of $B$.
Moreover $m>0$ since each pair of points lies in two blocks. We count the size of the following set in two ways:
\begin{align*}
    X:=\{(p,B')\mid p\in P_{1}\cup P_{2}, B'\in\Bmc, \text{ and } p\subset B'\}.
\end{align*}
It follows from \eqref{eq:d2-s1s2} that $|P_1\cup P_2| =c^{2}(c-1)$. Also each $p\in P_1\cup P_2$ occurs in exactly two
pairs in $X$ since $\Dmc$ is a biplane, so $o(x)=2c^2(c-1)$. On the other hand, each block $B$ occurs in exactly $m$ pairs in $X$, and so  $2c^{2}(c-1)=o(x)=m|\Bmc|=mv$, and since $v=c^{2}$, we have
$m=2(c-1)$.

(b)  Define a bipartite graph $\Gmc(B)$ with vertex set $\Gamma\times\{1,2\}$, and with edges all the pairs of the form
$\{(\alpha,1), (\beta,2)\}$ for which the point $(\alpha,\beta)\in B$. So $\Gmc(B)$ has exactly $k$ edges. Each pair from $P_1$ has the form
$\{(\alpha,\beta_1), (\alpha,\beta_2)\}$ with $\beta_1\ne \beta_2$, and determines  two edges of $\Gmc(B)$, namely $\{(\alpha,1), (\beta_i,2)\}$ for $i=1, 2$. Thus $(\alpha, 1)$ has valency at least $2$ in $\Gmc(B)$. Similarly a $3$-subset $S\subset B$ such that all pairs from $S$ lie in $P_1$ has the form $S=\{(\alpha,\beta_1), (\alpha,\beta_2), (\alpha,\beta_3)\}$ with the $\beta_i$ pairwise distinct, and determines three edges of $\Gmc(B)$, namely $\{(\alpha,1), (\beta_i,2)\}$, for $1\leq i\leq 3$. Conversely each vertex $(\alpha,1)$ of valency at least $3$ in $\Gmc(B)$ gives rise to at least one such 3-subset $S$. Similarly a
$3$-subset $S\subset B$ such that all pairs from $S$ lie in $P_2$ determines a vertex $(\beta,2)$ of valency at least 3 in $\Gmc(B)$, and each such vertex determines at least one such $3$-subset $S$.  Suppose that part (b) is false. Then each vertex of $\Gmc(B)$ has valency at most $2$. For $i, j\in \{ 1,2\}$, let $k_{ij}$ be the number of vertices of $\Gamma\times\{i\}$ of valency $j$. Then for each $i$,  $k=k_{i1}+2k_{i2}\geq 2k_{i2}$, and from our discussion above, $k_{i2}=|P_i|$. Thus
\[
2k\geq 2k_{12}+2k_{22}=  2|P_1|+2|P_2| = 2 c^2(c-1).
\]
However, by Lemma~\ref{lem:basic}, $k<\sqrt{2}\, c+1$ with $c\geq4$, and hence $\sqrt{2}\, c+1 > c^2(c-1)$, which is a contradiction. This proves part (b).

(c) Consider $L$ as in Hypothesis~\ref{hypo}. If $L=1$ there is nothing to prove, so assume that $L\ne 1$. Let $(x,1)\in L$ with $x\ne 1$. Then the fixed points $(\alpha,\beta)$ of $(x,1)$ in $\Pmc$ are precisely those for which $\alpha$ is a fixed point for $x$ in $\Gamma$. Thus $(x,1)$ fixes $|\fix_{\Gamma}(x)|c$ points of $\Pmc$
By Corollary~\ref{cor:fix-bi-bound}, the number of fixed points of $(x,1)$ is at most $k+\sqrt{k-2}< 2k-2$, and $\eqref{eq:kbound}$ implies that  $ 2k-2 <2\sqrt{2}\,c$. Hence $|\fix_{\Gamma}(x)|< 2\sqrt{2}<3$, so $|\fix_{\Gamma}(x)|\leq 2$, and part (c) follows.
%
\end{proof}

In particular $L\cong \sigma{L}\ne 1$. Now $\sigma(L)$ is a normal subgroup of the primitive group $H$, and hence $\sigma(L)$ is transitive on $\Gamma$.



\subsection{Automorphism groups of type  PA, HC, or CD}\label{sec:PA}

In this subsection, we prove that the group $G$  in Hypothesis~\ref{hypo} is not of type \PA, \HC\ or \CD, where these types are as in Table~\ref{tbl:prim}.

\begin{theorem}\label{thm:main-pa}
Let $\Dmc=(\Pmc,\Bmc)$ be a biplane admitting a point-primitive automorphism group $G$. Then $G$ cannot be of type \PA, \HC, or \CD.
\end{theorem}

If a biplane $\Dmc$ admits a group of one of these types, then  in particular {\rm Hypothesis \ref{hypo}} holds, and in Lemma~\ref{lem:socles} (which follows from the O'Nan--Scott Theorem -- see \cite[Chapter 7.6]{b:Praeger-18}), we see that the cartesian decomposition can be chosen to give important extra information about the group $L$.

\begin{lemma}\label{lem:socles}
Suppose that $\Dmc=(\Pmc,\Bmc)$ is a biplane admitting a point-primitive automorphism group $G$  of type  $\PA, \HC$, or $\CD$.  Then
$G$ preserves a cartesian decomposition such that {\rm Hypothesis \ref{hypo}} holds, and in addition, 
$\Soc(H)$ is neither regular, nor a Frobenius group on $\Gamma$, $\Soc(G)=\Soc(H)^d$, and $L\geq\Soc(H)\times 1^{d-1}$.
\end{lemma}

By Lemma~\ref{lem:socles}, $\Soc(H)$ is neither regular nor Frobenius, and hence, in particular, $L$ is not semiregular on $\Gamma$. Thus, by Lemma~\ref{lem:main}, $d=2$, and so by Lemma~\ref{lem:threepoints}, the stabiliser in $\Soc(H)$  of any three points is trivial. This group theoretic information will be crucial to our analysis. We will make use of the following information about primitive groups for which a stabiliser has a very small orbit.

\begin{lemma}\cite[Theorem]{a:Wong67}\label{lem:subdeg3}
Let $H$ be a primitive permutation group on $\Gamma$ with $\Soc(H)$ non-abelian, and such that $H_\alpha$ has an orbit of length $3$, where $\alpha\in\Gamma$. Then $(H,H_{\alpha})$ are as in {\rm Table~\ref{tbl:subs3}}.
\end{lemma}

\begin{table}
    \centering
    \small
    \caption{Primitive groups $H$ with non-abelian socle $T$ and a suborbit of length $3$.}\label{tbl:subs3}
    \begin{tabular}{lllllll}
        \hline
        $H$ &
        $|H_{\alpha}|$ &
        Degree &
        $T$ &
        $T_{\alpha}$ &
        Subdegrees of $T$ &
        Conditions\\
        \hline
        %
        $\PSL_{2}(q)$ 		&
        $\S_{4}$ &
        $\frac{q(q^{2}-1)}{48}$ &
        $\PSL_{2}(q)$ 		&
        $\S_{4}$ &
        -&
        $q\equiv \pm 1 \mod{16}$\\
        %
        $\PSL_{3}(3)$ &
        $\S_{4}$ &
        $234$ &
        $\PSL_{3}(3)$ &
        $\S_{4}$ &
        $1,3,4^{2},6,12^{8},24^{5}$ &
        \\
        %
        $\PSL_{3}(3){:}2$ 		&
        $\S_{4}{\times} \C_{2}$ &
        $234$ &
        $\PSL_{3}(3)$ 		&
        $\S_{4}$ &
        $1,3,4^{2},6,12^{8},24^{5}$&
        \\
        %
        $\PSL_{2}(13)$ 		&
        $\D_{12}$ &
        $91$ &
        $\PSL_{2}(13)$ 		&
        $\D_{12}$ &
        $1,3^2, 6^4, 12^5$
        &
        \\
        %
        $\PSL_{2}(11)$ 		&
        $\D_{12}$ &
        $55$ &
        $\PSL_{2}(11)$ 		&
        $\D_{12}$ &
        $1,3^2, 6^4, 12^2$
        &
        \\
        %
        $\PGL_{2}(7)$ 		&
        $\D_{12}$ &
        $28$ &
        $\PSL_{2}(7)$ 		&
        $\S_{3}$ &
        $1, 3^{3},6^{3}$
        \\
        %
        $\S_{5}$ 		&
        $\D_{12}$ &
        $10$  &
        $\A_{5}$ 		&
        $\S_{3}$ &
        $1$,$3$,$6$ \\
        %
        $\A_{5}$ 		&
        $\S_{3}$ &
        $10$ &
        $\A_{5}$ 		&
        $\S_{3}$ &
        $1$,$3$,$6$ \\
        \hline
    \end{tabular}
\end{table}


\begin{proposition}\label{prop:stab-three}
    Let $H$ be a finite primitive permutation group on $\Gamma$, with  socle $T$, such that $T_{\alpha,\beta,\gamma}=1$ for all pairwise distinct $\alpha,\beta,\gamma\in\Gamma$, and such that $T$ is neither regular nor a Frobenius group. Let $\alpha\in\Gamma$. Then
    \begin{enumerate}[\rm (a)]
        \item $T_\alpha$ has no orbit of length $2$.
        \item If $T_\alpha$ has an orbit of length $3$, then $T_{\alpha}=\S_{3}$, and  $(H,H_{\alpha}, |\Gamma|)$ is either $(\A_{5},\S_{3}, 10)$ or
        $(\S_{5},\D_{12}, 10)$.
        \item If $T_\alpha$ has an orbit of length $4$, then $T_{\alpha}=\A_{4}$, and the triple $(H,H_{\alpha}, |\Gamma|)$ is either $(\A_{5},\A_{4}, 5)$
        or $(\S_{5},\S_{4}, 5)$.
    \end{enumerate}
\end{proposition}



\begin{proof}
    Since $T$ is not regular on $\Gamma$, $T$ is nonabelian, and moreover  it follows from \cite[Corollary 2.2]{a:wielandt} that $\fix_{\Gamma}(T_\alpha)=\{\alpha\}$. Let $\Delta$ be a $T_\alpha$-orbit in $\Gamma\setminus\{\alpha\}$ of least cardinality.
    Then $|\Delta|\geq 2$, and we suppose that $|\Delta|$ is $2$, $3$ or $4$. Since the identity is the only element of $T$ fixing any three points of $\Gamma$, it follows that $T_\alpha$ acts faithfully on $\Delta$, that is, $T_\alpha^\Delta\cong T_\alpha$.
    If $T_{\alpha}^\Delta$ were regular, then this would imply that $|T_\alpha|=|\Delta|$, and then by the minimality of $|\Delta|$, all
    $T_\alpha$-orbits in  $\Gamma\setminus\{\alpha\}$ must have length $|\Delta|$. This means that $T_\alpha$ acts regularly and faithfully on  each of them, whence $T$ is a Frobenius group, which is a contradiction.   Thus $T_{\alpha}^\Delta$ is not regular, and
     in particular, $|\Delta|>2$, proving  part (a).

    If $|\Delta|=4$ and $T_{\alpha}^\Delta\cong \D_8$ or $\S_4$, then some nontrivial element of $T_\alpha$ fixes two points of $\Delta$, and hence fixes three points of $\Gamma$, which is a contradiction. The only possiblities remaining for $(|\Delta|, T_\alpha)$ are $(3, \S_3)$ and $(4, \A_4)$.

     Suppose first that  $(|\Delta|, T_\alpha)=(3, \S_3)$. Then all $T_\alpha$-orbits in $\Gamma\setminus\{\alpha\}$ have lengths 3 or 6. If there is a second orbit $\Delta'$ of length 3, then  for $\beta\in\Delta$, $T_{\alpha\beta}\ne 1$ and fixes at least three points ($\alpha$ and one in each of $\Delta$ and $\Delta'$), a contradiction. Thus $\Delta$ is the unique $T_\alpha$-orbit of length 3, which means that it is also an $H_\alpha$-orbit.
     By Lemma~\ref{lem:subdeg3},  $(H, T, |\Gamma|)$ appears in Table~\ref{tbl:subs3}, with $T_\alpha=\S_3$ and a unique subdegree 3, and hence is one of the triples in part (b).

     Suppose finally that $(|\Delta|, T_\alpha)=(4, \A_4)$. Then all $T_\alpha$-orbits in $\Gamma\setminus\{\alpha\}$ have lengths 4, 6 or 12.
The same argument as in the previous paragraph shows that $\Delta$ is the unique $T_\alpha$-orbit of length 4, and is also an $H_\alpha$-orbit. If $T_\alpha$ had an orbit $\Delta'$ of length $6$, then the action of $T_{\alpha}$ on $\Delta'$ is equivalent to its action by right multiplication on the cosets of $\langle (1,2)(3,4)\rangle$, and so an involution in $T_\alpha$ would fix two points of $\Delta'$ and hence three points of $\Gamma$, a contradiction.Thus each $T_\alpha$-orbit in $\Gamma\setminus (\{\alpha\}\cup\Delta)$
    has length 12, and in particular each involution in $T_\alpha$ fixes only the point $\alpha$. Hence $|\Gamma|$ is odd,
and $K:= \Obf_2(T_\alpha)\cong \C_2^2$ is a Sylow $2$-subgroup of $T$. Moreover, $\Nbf_T(K)$ fixes the unique
    fixed point $\alpha$ of $K$, so $\Nbf_T(K)=T_\alpha$. This implies that $K=\Cbf_T(K)$ is self-centralising.
    Since the  socle $T$ of $H$ is not regular it follows that $T=S^k$ for some non-abelian simple group $S$ and integer $k\geq1$. Since also $4$ divides $|S|$ while $|T|$ is not divisible by $8$ it follows that $k=1$, so $T$ is a non-abelian simple group.
    All simple groups $T$ with a self-centralising Sylow $2$-subgroup of order $4$ are known
    (see \cite[Theorem 15.2.1]{b:Gorenstein}), namely $T = \PSL_{2}(q)$ with $q\equiv 3,5 \pmod{8}$ and $q>3$.
    Let $x\in T_\alpha$ be an involution. Then $x$ has $f=1$ fixed point, and the stabiliser
    $T_\alpha$ contains exactly $u_1=3$ conjugates of $x$ while $T$ contains $u=|T:\Cbf_T(x)|$ conjugates of $x$.
    Applying Lemma~\ref{lem:conj}, we find
    \[
    \frac{|T|}{12} = |\Gamma|=\frac{|\Gamma|}{f}=  \frac{|T:\Cbf_T(x)|}{3}
    \]
    whence $|\Cbf_T(x)|=4$. However, for $T = \PSL_{2}(q)$, we have
    $\Cbf_T(x) = \D_{q\pm 1}$, and hence $q=5$ and
    $(H,H_{\alpha}, |\Gamma|)$ is $(\A_{5},\A_{4}, 5)$ or $(\S_{5},\S_{4}, 5)$, proving (c).
\end{proof}

We now prove Theorem \ref{thm:main-pa}.

\begin{proof}[\bf Proof of Theorem \ref{thm:main-pa}]
Suppose that $\Dmc=(\Pmc,\Bmc)$ is a biplane admitting a point-primitive automorphism group $G$ of type \PA, \HC,  or \CD.
Then by Lemma~\ref{lem:socles}, we may assume that Hypothesis~\ref{hypo} holds with $T=\Soc(H)$ neither regular nor Frobenius, and with $\Soc(G)=T^d$ and $L\geq T\times 1^{d-1}$. Since $T$ is not regular on $\Gamma$ it follows that $L$ is not semiregular on $\Pmc$.
Then by Lemma~\ref{lem:main},  $d=2$ and $c\geq11$.

Let $B$ be a block of $\Dmc$, say $B=\{(\alpha_{i},\beta_{i})\mid i=1,\ldots,k\}$. If $\alpha_1=\dots = \alpha_k$, then each of the $\binom{k}{2}=k(k-1)/2 = c^2-1$ unordered pairs from $B$ lies in the set $P_1$ of \eqref{eq:d2-s1s2}, while by Lemma \ref{lem:threepoints},
the number of such pairs is at most $2(c-1)$.
Hence $c^2-1\leq 2(c-1)$, which is a contradiction. The same argument shows that the $\beta_i$ are not all equal. On the other hand, by Lemma~\ref{lem:threepoints}, the block $B$ contains at least three points with the same first coordinate or at least three points with  the same second coordinate. For $i=1, 2$, let $r_{i}$ be the maximum size of a subset $S_i$ of $B$ for which the $i$th coordinates of all the points in $S_i$ are equal. We have just shown that $1\leq r_i\leq k-1$ for each $i$. Without loss of generality we may assume that $r_1\geq r_2$, and hence $r_1\geq3$. Note that $|S_1\cap S_2|\leq 1$, so we may further assume that   $\alpha_{1}=\ldots=\alpha_{r_{1}}=\alpha$ and $\beta_{r_{1}+\delta}=\ldots=\beta_{r_{1}+r_{2}+\delta}=\beta$, where $\delta\in\{0,1\}$. Finally, if $r_2=1$, then we may assume (since $r_1\leq k-1$) that $\delta=0$ and $S_2=\{(\alpha_{r_1+1},\beta)\}$ with $\alpha_{r_1+1}\ne \alpha$. Thus in all cases, $\alpha_{r_1+1}\ne\alpha$.

Write $\Soc(G)=T_1\times T_2$ with $T_i\cong T=\Soc(H)$ so that $T_1\leq L$.
Then $(T_{1})_{\alpha}$ fixes pointwise the $r_1\geq 3$ points of $S_1\subseteq B$. Since $\Dmc$ is a biplane, this implies that
$(T_{1})_{\alpha}$ fixes $B$, and hence that
$B$ contains $S_2':= \{(\alpha_{r_{1}+1}^x,\beta) \mid x\in (T_{1})_{\alpha}\}$. Note that $S_2'\subseteq S_2$ so $r_2\geq |S_2'|=|\alpha_{r_1+1}^{(T_{1})_{\alpha}}|$. As we noted above, $T_1$ is neither regular nor a Frobenius group, and by Lemma~\ref{lem:threepoints} all  three-point stabilisers in $T_1$ are trivial. Hence the primitive permutation group $H$ satisfies the hypotheses of Proposition~\ref{prop:stab-three}. Then  since $c=|\Gamma|\geq 11$, it follows from Proposition~\ref{prop:stab-three} that every orbit of $(T_{1})_{\alpha}$ in $\Gamma\setminus\{\alpha\}$ has size at least $5$. Thus $r_1\geq r_2\geq |\alpha_{r_1+1}^{(T_{1})_{\alpha}}|\geq 5$,
 and if $(\alpha,\beta)\in B$, then even $r_2-1\geq  |\alpha_{r_1+1}^{(T_{1})_{\alpha}}|\geq 5$, so $r_1\geq r_2\geq 6$.

Since $r_2\geq 3$, we may argue as above with $1\times (T_{2})_{\beta}$ in place of $(T_1)\alpha\times 1$, and conclude that
also $(T_2)_\beta$ fixes $B$. Consider $R:=(T_{1})_{\alpha}\times (T_{2})_{\beta}$. We have shown
that $R\leq (T_{1}\times T_{2})_{B}$. We may partition $B$ as follows:
\begin{align}\label{eq:thm-PA}
    \nonumber    S_{1}&:=\{(\alpha,\beta_{i})\in B\mid 1\leq i\leq r_1\};\\
    S_{2}''&:=\{(\alpha_{i},\beta)\in B\mid \alpha_{i}\in \Gamma\}\setminus\{(\alpha,\beta)\}\subseteq S_2;\\
    \nonumber    S_{3}&:=B\setminus(S_{1}\cup S_{2}'').
\end{align}
Then $B$ is the disjoint union $S_{1}\cup S_{2}''\cup S_{3}$. Since $R=(T_{1})_{\alpha}\times (T_{2})_{\beta}$
fixes each of $S_1, S_2'', S_3$ setwise, each of these sets is a union of $R$-orbits, each of of these $R$-orbits has size at least 5.
Suppose that $O_{i}=(\alpha_{i},\beta_{i})^R$, for $i=1,\ldots,t$, are the $R$-orbits contained  $S_{3}$ (where $t\geq0$). Note that each
$\alpha_{i}\neq \alpha$ and each $\beta_{i}\neq \beta$.
Set $r_{1i}=|\alpha_{i}^{(T_{1})_{\alpha}}|$ and $r_{2i}=|\beta_{i}^{(T_{2})_{\beta}}|$. Then,
\begin{align*}
|S_{1}|&=r_{1},\\
|S_{2}''|&=
\begin{cases}
r_{2}-1,& \hbox{if $(\alpha,\beta)\in B$}; \\
r_{2},&\hbox{otherwise.}\\
\end{cases}\\
|S_{3}|&=\sum_{i=1}^{t}r_{1i}r_{2i},
\end{align*}
and each of $r_{1}$, $|S_2''|$,  $r_{1i}$ and $r_{2i}$ (for  $i=1,\ldots,t$) is at least $5$.

We now determine the number of ordered pairs $(p,q)\in B\times B$  such that $p\ne q$, $\{p, q\}$ is contained in $S_1$ or $S_2''$ or  the same $R$-orbit in $S_3$, and $p, q$
have either the same first coordinate, or the same second coordinate. The number of these pairs
with $p,q\in S_{1}$ is $r_{1}(r_{1}-1)$; the number of them with $p,q\in S_{2}''$ is $r_{2}(r_{2}-1)$
or $(r_{2}-1)(r_{2}-2)$ according as $(\alpha,\beta)\notin B$ or $(\alpha,\beta)\in B$, respectively.
Finally, for each $i=1,\dots,t$, the number of such pairs with $p,q\in O_{i}$ is $r_{1i}r_{2i}(r_{2i}-1)+r_{1i}r_{2i}(r_{1i}-1)=r_{1i}r_{2i}(r_{1i}+r_{2i}-2)$.
Note that if $t=0$, then there are no pairs of the last type. Hence, the total number $m$ of such ordered pairs satisfies
\begin{align*}
m=
\begin{cases}
r_{1}(r_{1}-1)+r_{2}(r_{2}-1)+r_{t}& \hbox{if $(\alpha,\beta)\notin B$,} \\
r_{1}(r_{1}-1)+(r_{2}-1)(r_{2}-2)+r_{t}&\hbox{if $(\alpha,\beta)\in B$,}\\
\end{cases}
\end{align*}
where
\begin{align*}
r_{t}=
\begin{cases}
\sum_{i=1}^{t} r_{1i}r_{2i}(r_{1i}+r_{2i}-2)& \hbox{if $S_{3}\neq \emptyset$}, \\
0&\hbox{if $S_{3}= \emptyset$.}\\
\end{cases}
\end{align*}

The fact that $r_{1i}$ and $r_{2i}$ are at least $5$, for all $i$, implies that $r_{t}\geq 8\sum_{i=1}^{t}r_{1i}r_{2i}\geq 4|S_{3}|$
(and the inequality $r_t\geq 4|S_3|$ holds also if $S_{3}= \emptyset$). Moreover, if $(\alpha,\beta)\notin B$, then $r_{1}(r_{1}-1)+r_{2}(r_{2}-1)
\geq 4(|S_{1}|+|S_{2}''|)$ as $r_{1}\geq 5$ and $r_{2}\geq 5$. On the other hand, if $(\alpha,\beta)\in B$, then $r_2-1=|S_2''|\geq5$ and again
$r_{1}(r_{1}-1)+(r_{2}-1)(r_{2}-2)\geq 4(r_{1}+r_{2}-1)= 4(|S_{1}|+|S_{2}''|)$. Therefore, in all cases $m\geq 4(|S_{1}|+|S_{2}''|+|S_{3}|)=4k$.
However, by Lemma \ref{lem:threepoints}, the total number of ordered pairs of points of $B$ with the same first coordinate or the same second coordinate is $4(c-1)$.
Therefore, $m\leq 4(c-1)$, and so $4k\leq m\leq 4(c-1)$. Thus  $k< c$, contradicting Lemma \ref{lem:basic}.
\end{proof}

\subsection{Automorphism groups preserving a cartesian decomposition}\label{sec:car}

In this section, we complete the proof of Theorem~\ref{thm:main}. Let $\Dmc$ be a biplane admitting a point-primitive group $G$ which preserves a cartesian decomposition. Then the type of $G$ is not \HS\ or \SD\ (as these groups do not preserve cartesian decompositions), and is not \PA, \HC, or \CD, by Theorem~\ref{thm:main-pa}. The remaining types are \HA, \AS, and \TW.  First we show that type \HA\ is possible by giving an example of a biplane with parameters  $(16,6,2)$ admitting an automorphism group preserving a homogeneous cartesian decomposition. Then we prove in Proposition~\ref{prop:AS} that type \AS\ is not possible. This will complete the proof of Theorem~\ref{thm:main}.

\begin{example}\label{ex:S4wrS2}
For $i=1,\ldots,5$, let $\alpha_{i}\in\S_{16}$, acting on  $\Pmc=\{1,\ldots,16\}$, defined as in \eqref{eq:ex}, and let $G=\langle \alpha_{i}\mid i=1,\ldots,5\rangle$.
\begin{align}\label{eq:ex}
\nonumber \alpha_{1}:=&(2,4,3)(5,13,9)(6,16,11)(7,14,12)(8,15,10);\\ \nonumber \alpha_{2}:=&(2,6,5)(3,11,9)(4,16,13)(7,12,14)(8,15,10);\\ \alpha_{3}:=&(2,6)(3,11)(4,16)(7,15)(8,12)(10,14);\\
\nonumber \alpha_{4}:=&(3,5)(4,6)(11,13)(12,14);\\
\nonumber    \alpha_{5}:=& (1,2)(3,4)(5,6)(7,8)(9,10)(11,12)(13,14)(15,16).
\end{align}
Then  $G\cong \S_{4}\wr \S_{2}$, and is a primitive permutation group of rank $3$ of affine type \HA, on $16$ points, and is a maximal subgroup of index $10$ of a larger primitive group $X:= \C_{2}^{4}{:}\S_{6}$.   The point-stabiliser $G_{1}$ of $G$ is isomorphic to $(\S_{3}\times \S_{3}){:}\C_{2}$. Let $B=\{1, 2, 3, 5, 9, 16\}$. Then the set $\Bmc=B^{G}$ of $G$-translates of $B$  forms the block set for a biplane $\Dmc=(\Pmc,\Bmc)$ with parameters $(16,16,2)$. The full automorphism group of $\Dmc$ is the group $X$. The biplane $\Dmc$ is $G$-flag-transitive with block-stabiliser isomorphic to $\S_{3}$. We moreover observe that the  set $\Emc=\{\Gamma_{1},\Gamma_{2}\}$ defined in \eqref{eq:ex-cd} is a homogeneous cartesian decomposition of the  point set $\Pmc$ and the  group $G$ preserves this cartesian decomposition. (Each point $i$ is identified with the unique ordered pair $(\gamma_1,\gamma_2)\in \Gamma_1\times\Gamma_2$ where $\gamma_j$  is the part of $\Gamma_j$ containing $i$.)
\begin{align}\label{eq:ex-cd}
\nonumber    \Gamma_{1}:=\{ &\{ 1, 8, 10, 15 \}, \{ 2, 7, 9, 16 \}, \{ 3, 6, 12, 13 \}, \{ 4, 5, 11, 14 \} \};\\
\Gamma_{2}:=\{ &\{ 1, 7, 12, 14 \}, \{ 2, 8, 11, 13 \}, \{ 3, 5, 10, 16 \}, \{ 4, 6, 9, 15 \} \}.
\end{align}	
The full automorphism group $X$ does not preserve this cartesian decomposition.
\end{example}

For the proof of Proposition \ref{prop:AS}, we need to know the possible solutions of the Diophantine equation \eqref{eq:dio}:

\begin{lemma}\label{lem:dio}
The Diophantine equation
\begin{align}\label{eq:dio}
    8x^2-y^2 = 7
\end{align}
has infinitely many solutions $(x_n,y_n)_{n\geq 0}$ and $(x'_n,y'_n)_{n\geq 0}$ in positive integers, where
\begin{align}\label{eq:dio-sol}
\begin{cases}
    x_n=u_n+v_n\\
    y_n=u_n+8v_n
\end{cases}
    \mbox{ and }
\begin{cases}
    x'_n=u_n-v_n\\
    y'_n=-u_n+8v_n,
\end{cases}
\end{align}
and $(u_n,v_n)_{n\geq 0}$ is defined in \eqref{eq:dio-sol-uv} below:
\begin{align}\label{eq:dio-sol-uv}
\begin{cases}
    u_n=\dfrac{1}{2}\left[
    \left(
    3+\sqrt{8}\right)^n
    +
    \left(
    3-\sqrt{8}\right)^n
    \right]\\
    \\
    v_n=\dfrac{1}{2\sqrt{8}}\left[
    \left(
    3+\sqrt{8}\right)^n
    -
    \left(
    3-\sqrt{8}\right)^n
    \right].
\end{cases}
\end{align}
Moreover, the pairs $(x_n,y_n)_{n\geq 0}$ and $(x'_n,y'_n)_{n\geq 0}$ are the only solutions to \eqref{eq:dio}.
\end{lemma}

\begin{proof}
   This follows from \cite[Theorem 4.1.4]{b:Dio} and the following argument on page 60 of that book.
\end{proof}

\begin{proposition}\label{prop:AS}
 Suppose that {\rm Hypothesis \ref{hypo}} holds. Then $G$ cannot be of type \AS.
\end{proposition}
\begin{proof}
Suppose that $\Soc(G)=T$ with $T$ a non-abelian simple group. Since $G$ preserves the cartesian decomposition $\Pmc=\Gamma^d$, it follows from \cite[Theorem 8.21]{b:Praeger-18} that $d=2$ and $(T, v)$ is one of the following:
\begin{enumerate}[\rm (i)]
    \item $T=\A_6$ and $v=36$;
    \item $T=\M_{12}$ and $v=144$;
    \item $T=\PSp_4(q)$ with $q\geq 4$ even, and $v=q^4(q^2-1)^2/4$.
\end{enumerate}
If $T=\A_6$ or $\M_{12}$, then $c$ is $6$ or $12$, respectively, and by Lemma \ref{lem:basic},   $k=\lceil\sqrt{2}\,c\rceil$ is $9$ or $17$, respectively. However, \eqref{eq:basic} then implies that $v$ is $37$ or $136$, respectively, which is a contradiction. Thus $T=\PSp_4(q)$ with $q=2^a\geq 4$ and $v=c^2$, where $c=q^2(q^2-1)/2$. By \eqref{eq:basic},  $k(k-1)=2(c^2-1)$. Solving the quadratic equation $k^2-k-2c^2+2=0$, we find that $k=(1+\sqrt{8c^2-7})/2$. Therefore, $8c^2-7=e^2$, for some positive integer $e$.
We now apply Lemma \ref{lem:dio}, and conclude that $c=u_n\pm v_n$, for some $n\geq 0$, where $u_n$ and $v_n$ are defined as in \eqref{eq:dio-sol-uv}. Since $q\geq 4$, it follows that $c\geq 120$, and so by inspecting the values of $u_n$ and $v_n$ for small values of $n$, we conclude that $n\geq 3$.

Let $a=3$ and $b=\sqrt{8}$. Then
\begin{align*}
u_n = \frac{1}{2}\sum_{i=0}^n\binom{n}{i}a^{n-i}b^{i}(1+(-1)^i) =  \sum_{0\leq i\leq n,\, i\ {\rm even}}\binom{n}{i}a^{n-i}b^{i}
\end{align*}
and
\begin{align*}
v_n = \frac{1}{2b}\sum_{i=0}^n\binom{n}{i}a^{n-i}b^{i}(1-(-1)^i) =  \sum_{1\leq i\leq n,\, i\ {\rm odd}}\binom{n}{i}a^{n-i}b^{i-1}
\end{align*}
and we note that each term in each of these sums is a positive integer. We evaluate $u_n, v_n$ modulo $3$ and find, noting that $a=3$, that

\begin{center}
$u_n\equiv 0\pmod{3}$ if $n$ is odd, and  $u_n\equiv b^n=2^{3n/2}\pmod{3}$ if $n$ is even, while\\
$v_n\equiv b^{n-1}=2^{3(n-1)/2}\pmod{3}$ if $n$ is odd, and  $v_n\equiv 0\pmod{3}$ if $n$ is even.
\end{center}
Now $c=u_n + \epsilon v_n$ where
$\epsilon=\pm 1$. Thus $c\equiv \epsilon 2^{3(n-1)/2}\pmod{3}$ if $n$ is odd, and     $c\equiv  2^{3n/2}\pmod{3}$ if $n$ is even. This implies that, for all $n$ and $\epsilon$, $c\equiv \pm 1\pmod{3}$. However $3$ divides $q^2-1$ for all even $q$, and hence $3$ divides $c$. This contradiction completes the proof.
\end{proof}

\begin{proof}[\bf Proof of Theorem \ref{thm:main}]
By \cite[Theorems 8.4 and 9.19]{b:Praeger-18}\label{lem:inclusion}, if $G$ is a primitive group preserving a homogeneous decomposition, then is one of the types \HA, \AS, \HC, \CD, \TW\ or \PA. In particular, $G$ is neither of type \HS, nor \SD. Moreover, Theorem~\ref{thm:main-pa} and Proposition~\ref{prop:AS} imply that $G$ can only have type $\HA$ or $\TW$.
\end{proof}

\section*{Acknowledgements}

The third author acknowledges support from Australian Research Council Discovery Project Grant DP200100080.
The first and second  authors are also grateful to Cheryl E. Praeger and Alice Devillers  for supporting their visit to The University of Western Australia during July-September 2019.


\end{document}